\documentclass[reqno]{amsart}
\usepackage{amsmath, amssymb, amsthm, epsfig}
\usepackage{hyperref, latexsym}
\usepackage{url}
\usepackage[mathscr]{euscript}
\usepackage{ stmaryrd }

\usepackage{color}
\usepackage{fullpage} 
\usepackage{setspace}
\usepackage{microtype}

\def\today{\ifcase\month\or=
  January\or February\or March\or April\or May\or June\or=
  July\or August\or September\or October\or November\or December\fi=
  \space\number\day, \number\year}

 \newtheorem{theorem}{Theorem}
  
 \newtheorem{lemma}[theorem]{Lemma}

 \theoremstyle{definition}

 \theoremstyle{remark}

 \newcommand{\mc}{\mathcal}

 \newcommand{\M}{\mc{M}}
\newcommand{\T}{\mathbb{T}}

 \newcommand{\R}{\mathbb{R}}

 \newcommand{\Z}{\mathbb{Z}}

 \newcommand{\ds}{\text{\rm d}s}

 \newcommand{\du}{\text{\rm d}u}

 \newcommand{\dx}{\text{\rm d}x}
 
 \newcommand{\dy}{\text{\rm d}y}

    \renewcommand{\d}{\text{\rm d}}

\newcommand{\ov}{\overline}
\renewcommand{\H}{\mc{H}}

\renewcommand{\S}{\mathbb{S}}
\newcommand{\pd}[2]{\frac{\partial #1}{\partial #2}}

\begin{document}

\title[Variation of maximal operators]{On the variation of maximal operators \\ of convolution type II}
\author[Carneiro, Finder and Sousa]{Emanuel Carneiro, Renan Finder and Mateus Sousa}
\date{\today}

\address{IMPA - Instituto de Matem\'{a}tica Pura e Aplicada, Estrada Dona Castorina 110, Rio de Janeiro - RJ, Brazil, 22460-320.}
\email{carneiro@impa.br}
\email{feliz@impa.br}
\email{mateuscs@impa.br}

\allowdisplaybreaks
\numberwithin{equation}{section}

\maketitle

\begin{abstract} In this paper we establish that several maximal operators of convolution type, associated to elliptic and parabolic equations, are variation-diminishing. Our study considers maximal operators on the Euclidean space $\R^d$, on the torus $\T^d$ and on the sphere $\S^d$. The crucial regularity property that these maximal functions share is that they are subharmonic in the corresponding detachment sets.

\end{abstract}

\section{Introduction}

\subsection{Background} Let $\varphi: \R^d \times (0,\infty) \to \R$ be a nonnegative function such that \begin{equation}\label{Intro_eq_integral_1}
\int_{\R^d} \varphi(x, t)\,\dx = 1
\end{equation} 
for each $t >0$. Assume also that, when $t \to 0$, the family $\varphi(\cdot, t)$ is an approximation of the identity, in the sense that $\lim_{t \to 0} \varphi(\cdot, t) * f (x) = f(x)$ for a.e. $x \in \R^d$, if $f \in L^p(\R^d)$ for some $1 \leq p \leq \infty$. For an initial datum $u_0:\R^d \to \R$ we consider the evolution $u: \R^d \times (0,\infty) \to \R$ given by
$$u(x,t) = \varphi(\cdot, t) * |u_0| (x),$$
and the associated maximal function 
\begin{equation*}
u^*(x) = \sup_{t>0} \,u(x,t).
\end{equation*}
For a fixed time $t>0$, due to \eqref{Intro_eq_integral_1}, the convolution $\varphi(\cdot, t) *|u_0|$ is simply a weighted average of $|u_0|$, and hence it does not increase its variation (understood as the classical total variation or, more generally, as an $L^p$-norm of the gradient for some $1 \leq p \leq \infty$). One of the questions that interest us here is to know whether this smoothing behavior is preserved when we pass to the maximal function $u^*$. For instance, if $u_0 :\R \to \R$ is a function of bounded variation, do we have
\begin{equation}\label{Intro_variation_diminishing_problem}
V(u^*) \leq C\,\,V(u_0)
\end{equation}
with $C = 1$? Here $V(f)$ denotes the total variation of the function $f$.

\smallskip

The most natural example of an operator in this framework is the Hardy-Littlewood maximal operator, in which $\varphi(x,t) = \frac{1}{t^d m(B_1)}\chi_{B_1}(x/t)$, where $B_1$ is the unit ball centered at the origin and $m(B_1)$ is its $d$-dimensional Lebesgue measure. In this case, due to the work of Kurka \cite{Ku}, the one-dimensional estimate \eqref{Intro_variation_diminishing_problem} is known to hold with constant $C =  240,004$, but the problem with $C=1$ remains open. For the one-dimensional right (or left) Hardy-Littlewood maximal operator, i.e. when $\varphi(x,t) = \frac{1}{t}\chi_{[0,1]}(x/t)$, estimate \eqref{Intro_variation_diminishing_problem} holds with $C=1$ due to the work of Tanaka \cite{Ta}. The sharp bound \eqref{Intro_variation_diminishing_problem} with constant $C=1$ also holds for the one-dimensional uncentered version of this operator, as proved by Aldaz and P\'{e}rez L\'{a}zaro \cite{AP}. Higher dimensional analogues of \eqref{Intro_variation_diminishing_problem} for the Hardy-
Littlewood maximal operator, centered or uncentered, are open problems (see, for instance, the work of Haj\l asz and Onninen \cite{HO}). Other interesting works related to the regularity of the Hardy-Littlewood maximal operator and its variants, when applied to Sobolev and BV functions, are \cite{BCHP, CH, CMa, CM, HM, Ki, KL, KiSa, Lu1, Lu2, St, Te}.

\smallskip

In the precursor of this work \cite[Theorems 1 and 2]{CS}, Carneiro and Svaiter proved the variation-diminishing property, i.e. inequality \eqref{Intro_variation_diminishing_problem} with $C=1$, for the maximal operators associated to the Poisson kernel 
\begin{equation}\label{Intro_Poisson}
P(x,t) = \frac{\Gamma \left(\frac{d+1}{2}\right)}{\pi^{(d+1)/2}}\ \frac{t}{(|x|^2 + t^2)^{(d+1)/2}}
\end{equation}
and the Gauss kernel
\begin{equation}\label{Intro_Gauss}
K(x,t) = \frac{1}{(4 \pi t)^{d/2}}\ e^{-|x|^2/4t}.
\end{equation}
Their proof is based on an interplay between the analysis of the maximal functions and the structure of the underlying partial differential equations (Laplace's equation and heat equation). The aforementioned examples are the only maximal operators of convolution type for which inequality \eqref{Intro_variation_diminishing_problem} has been established (even allowing a constant $C>1$).

\subsection{Maximal operators associated to elliptic equations}

A question that derives from our precursor \cite{CS} is whether the variation-diminishing property is a peculiarity of the smooth kernels \eqref{Intro_Poisson} and \eqref{Intro_Gauss} or if these can be seen as particular cases of a general family. One could, for example, look at the semigroup structure via the Fourier transforms \footnote{Our normalization of the Fourier transform is $\widehat{f}(\xi) = \int_{\R^d} e^{-2 \pi i x \cdot \xi}\,f(x)\,\dx$.} (in space) of these kernels:
\begin{equation*}
\widehat{P}(\xi,t) = e^{-t (2\pi |\xi|)} \ \ \ \ {\rm and} \ \ \ \ \ \widehat{K}(\xi,t) =  e^{-t (2\pi |\xi|)^2}.
\end{equation*}
A reasonable way to connect these kernels would be to consider the one-parameter family 
\begin{equation*}
\widehat{\varphi}_{\alpha}(\xi,t) =  e^{-t (2\pi |\xi|)^{\alpha}},
\end{equation*}
for $1 \leq \alpha \leq 2$.
However, in this case, the function $u(x,t) = \varphi_{\alpha}(\cdot, t)*u_0(x)$ solves an evolution equation related to the fractional Laplacian
$$u_t + (-\Delta)^{\alpha/2}\,u = 0\,,$$
for which we do not have a local maximum principle, essential to run the argument of Carneiro and Svaiter in \cite{CS}. The problem of proving that the corresponding maximal operator is variation-diminishing seems more delicate and it is currently open.

\smallskip

A more suitable way to address this question is to consider the Gauss kernel as an appropriate limiting case. For $a>0$ and $b \geq 0$ we define (motivated by the partial differential equation \eqref{Ell_equation} below)
\begin{equation}\label{Intro_Def_phi_hat_a_b}
\widehat{\varphi}_{a,b}(\xi, t) := e^{-t \left(\frac{-b + \sqrt{b^2 + 16 a \pi^2 |\xi|^2}}{2a}\right)}.
\end{equation}
Note that when $a=1$ and $b=0$ we have the Fourier transform of the Poisson kernel, and when $b=1$ and $a \to 0^+$ the function \eqref{Intro_Def_phi_hat_a_b} tends pointwise to the Fourier transform of the Gauss kernel by a Taylor expansion. For completeness, let us then define 
\begin{equation}\label{Intro_Def_phi_hat_0_b}
\widehat{\varphi}_{0,b}(\xi, t) := e^{-\frac{t}{b} (2\pi |\xi|)^2},
\end{equation}
for $b>0$. We will show that the inverse Fourier transform 
\begin{equation}\label{def_kernel_phi_a_b}
{\varphi}_{a,b}(x, t) = \int_{\R^d} \widehat{\varphi}_{a,b}(\xi, t) \,e^{2 \pi i x \cdot \xi}\,\d\xi
\end{equation}
is a {\it nonnegative radial function} that has the desired properties of an approximation of the identity. Let us consider the corresponding maximal operator
\begin{equation}\label{def_max_function_a_b}
u^*(x) = \sup_{t>0} {\varphi}_{a,b}(\cdot, t)*|u_0|(x).
\end{equation}
The fact that $u^*(x) \leq Mu_0(x)$ pointwise, where $M$ denotes the Hardy-Littlewood maximal operator, follows as in \cite[Chapter III, Theorem 2]{S}. Hence, for $1 < p \leq \infty$, we have $\|u^*\|_{L^p(\R^d)} \leq C\, \|u_0\|_{L^p(\R^d)}$ for some $C>1$. We also notice, from the work of Kinnunen \cite[proof of Theorem 1.4]{Ki}, that the maximal operator of convolution type \eqref{def_max_function_a_b} is bounded on $W^{1,p}(\R^d)$ for $1< p \leq \infty$, with $\|\nabla u^*\|_{L^p(\R^d)} \leq C\, \|\nabla u_0\|_{L^p(\R^d)}$ for some $C>1$. 

\smallskip

Our first result establishes that the corresponding maximal operator \eqref{def_max_function_a_b} is indeed {\it variation-diminishing} in multiple contexts. This extends \cite[Theorems 1 and 2]{CS}.

\begin{theorem}\label{Thm1}
Let $a,b\geq0$ with $(a,b) \neq (0,0)$, and let $u^*$ be the maximal function defined in \eqref{def_max_function_a_b}. The following propositions hold.
\smallskip
\begin{itemize}
\item[(i)] Let $1 < p \leq \infty$ and $u_0 \in W^{1,p}(\R)$. Then $u^{*} \in W^{1,p}(\R)$ and 
\begin{equation*}
\|(u^*)'\|_{L^p(\R)} \leq \|u_0'\|_{L^p(\R)}.
\end{equation*}

\vspace{0.15cm}

\item[(ii)] Let $u_0 \in W^{1,1}(\R)$. Then $u^{*} \in L^{\infty}(\R)$ and has a weak derivative $(u^*)'$ that satisfies 
\begin{equation*}
\|(u^*)'\|_{L^1(\R)} \leq \|u_0'\|_{L^1(\R)}.
\end{equation*}

\vspace{0.15cm}

\item[(iii)] Let $u_0$ be of bounded variation on $\R$. Then $u^*$ is of bounded variation on $\R$ and 
\begin{equation*}
V(u^*) \leq V(u_0).
\end{equation*}

\vspace{0.15cm}

\item[(iv)] Let $d>1$ and $u_0 \in W^{1,p}(\R^d)$, for $p =2$ or $p= \infty$. Then $u^{*} \in W^{1,p}(\R^d)$ and 
\begin{equation*}
\|\nabla u^*\|_{L^p(\R^d)} \leq \|\nabla u_0\|_{L^p(\R^d)}.
\end{equation*}

\end{itemize}
\end{theorem}

We shall see that the kernel \eqref{def_kernel_phi_a_b} has an elliptic character (when $a>0$) in the sense that $u(x,t) = {\varphi}_{a,b}(\cdot, t) * |u_0|(x)$ solves the equation 
\begin{equation}\label{Ell_equation}
au_{tt} - bu_t + \Delta u = 0 \ \ \ \ {\rm in}  \ \ \  \R^d \times (0,\infty)
\end{equation}
with
\begin{equation*}
\lim_{t \to 0^+} u(x,t) = |u_0(x)| \ \ \ \ {\rm a.e.}  \ \ {\rm in} \ \ \R^d.
\end{equation*}
In particular, the corresponding maximum principle plays a relevant role in our analysis. By appropriate dilations in the space variable $x$ and the time variable $t$, Theorem \ref{Thm1} essentially splits into three regimes: (i) the case $a=1$ and $b=0$ (which models all cases $a>0$ and $b=0$, corresponding to Laplace's equation) in which the level surfaces $|\xi| = \tau$ in \eqref{Intro_Def_phi_hat_a_b} are {\it cones}; (ii) the case $a=0$ and $b=1$ (which models all cases $a=0$ and $b>0$, corresponding to the heat equation), in which the level surfaces $|\xi|^2 = \tau$ in \eqref{Intro_Def_phi_hat_0_b} are {\it paraboloids}; (iii) the case $a=1$ and $b=1$ (which models all the remaining cases $a>0$ and $b>0$), in which the level surfaces $-1 + \sqrt{1 + 16 \pi^2 |\xi|^2} = \tau$ in \eqref{Intro_Def_phi_hat_a_b} are {\it hyperboloids}. The first two cases were proved in \cite[Theorems 1 and 2]{CS} (although here we provide a somewhat different and simpler proof than that of \cite{CS}) and the third regime is the novel contribution of this section. 

\subsection{Periodic analogues} We now address similar problems in the torus $\T^d \simeq \R^d/\Z^d$. For $a>0$, $b\geq 0$, $t>0$ and $n \in \Z^d$ let us now define
\begin{equation*}
\widehat{\Psi}_{a,b}(n, t) := e^{-t \left(\frac{-b + \sqrt{b^2 + 16 a \pi^2 |n|^2}}{2a}\right)},
\end{equation*}
and when $a=0$ and $b>0$ we define 
\begin{equation*}
\widehat{\Psi}_{0,b}(n, t) := e^{-\frac{t}{b} (2\pi |n|)^2}.
\end{equation*}
We then consider the periodic kernel, for $x \in \R^d$, 
\begin{equation*}
{\Psi}_{a,b}(x, t) = \sum_{n \in \Z^d} \widehat{\Psi}_{a,b}(n, t) \,e^{2 \pi i x \cdot n}.
\end{equation*}
It is clear that ${\Psi}_{a,b} \in C^{\infty}(\R^d \times (0,\infty))$. By Poisson summation formula, ${\Psi}_{a,b}$ is simply the periodization of   ${\varphi}_{a,b}$ defined in \eqref{def_kernel_phi_a_b}, i.e.
\begin{equation*}
{\Psi}_{a,b}(x, t) = \sum_{n \in \Z^d} {\varphi}_{a,b}(x +n, t).
\end{equation*}
Since ${\varphi}_{a,b}$ is nonnegative, and $\widehat{\Psi}_{a,b}(n, t)$ is also nonnegative, it follows that 
\begin{equation*}
0 \leq {\Psi}_{a,b}(x, t) \leq {\Psi}_{a,b}(0, t) 
\end{equation*}
for all $x \in \R^d$ and $t >0$. The approximate identity properties of the family $ {\varphi}_{a,b}(\cdot, t)$, reviewed in Section   \ref{Section2.1}, transfer to ${\Psi}_{a,b}(\cdot, t)$ in the periodic setting. For an initial datum $u_0:\T^d \to \R$ (which we identify with its periodic extension to $\R^d$) we keep denoting the evolution $u(x,t): \T^d \times (0,\infty) \to \R^+$ by 
\begin{equation}\label{Connecting_two_max_fun}
u(x,t) = {\Psi}_{a,b}(\cdot, t)*|u_0|(x) = \int_{\T^d} {\Psi}_{a,b}(x-y, t)\, |u_0(y)|\,\dy = \int_{\R^d} {\varphi}_{a,b}(x-y, t)\, |u_0(y)|\,\dy.
\end{equation}
Also, we keep denoting the maximal function $u^*:\T^d \to \R^+$ by
\begin{equation}\label{max_func_per}
u^*(x) = \sup_{t>0} \, u(x,t).
\end{equation}
From \eqref{Connecting_two_max_fun} it follows that $u^*(x) \leq Mu_0(x)$, where $M$ denotes the Hardy-Littlewood maximal operator on $\R^d$, and hence the operator $u_0 \mapsto u^*$ is bounded on $L^p(\T^d)$ for $1 < p \leq \infty$ and maps $L^1(\T^d)$ into $L^1_{weak}(\T^d)$ (the case $p=\infty$ is trivial; the case $p=1$ follows by the usual Vitali covering argument; the general case $1 < p < \infty$ follows by Marcinkiewicz interpolation). Then, it follows as in \cite[proof of Theorem 1.4]{Ki} that $u_0 \mapsto u^*$ is bounded on $W^{1,p}(\T^d)$ for $1 < p \leq \infty$, with $\|\nabla u^*\|_{L^p(\T^d)} \leq C\, \|\nabla u_0\|_{L^p(\T^d)}$ for some $C>1$.

\smallskip

Our second result establishes the variation-diminishing property for the operator \eqref{max_func_per} in several cases.

\begin{theorem}\label{Thm2}
Let $a,b\geq0$ with $(a,b) \neq (0,0)$, and let $u^*$ be the maximal function defined in \eqref{max_func_per}. The following propositions hold.
\smallskip
\begin{itemize}
\item[(i)] Let $1 < p \leq \infty$ and $u_0 \in W^{1,p}(\T)$. Then $u^{*} \in W^{1,p}(\T)$ and 
\begin{equation*}
\|(u^*)'\|_{L^p(\T)} \leq \|u_0'\|_{L^p(\T)}.
\end{equation*}

\vspace{0.15cm}

\item[(ii)] Let $u_0 \in W^{1,1}(\T)$. Then $u^{*} \in L^{\infty}(\T)$ and has a weak derivative $(u^*)'$ that satisfies 
\begin{equation*}
\|(u^*)'\|_{L^1(\T)} \leq \|u_0'\|_{L^1(\T)}.
\end{equation*}

\vspace{0.15cm}

\item[(iii)] Let $u_0$ be of bounded variation on $\T$. Then $u^*$ is of bounded variation on $\T$ and 
\begin{equation*}
V(u^*) \leq V(u_0).
\end{equation*}

\vspace{0.15cm}

\item[(iv)] Let $d>1$ and $u_0 \in W^{1,p}(\T^d)$, for $p =2$ or $p= \infty$. Then $u^{*} \in W^{1,p}(\T^d)$ and 
\begin{equation*}
\|\nabla u^*\|_{L^p(\T^d)} \leq \|\nabla u_0\|_{L^p(\T^d)}.
\end{equation*}

\end{itemize}
\end{theorem}

As in the case of $\R^d$, a relevant feature for proving Theorem \ref{Thm2} is the fact that $u(x,t) = {\Psi}_{a,b}(\cdot, t) * |u_0|(x)$ solves the partial differential equation 
\begin{equation*}
au_{tt} - bu_t + \Delta u = 0 \ \ \ \ {\rm in}  \ \ \  \T^d \times (0,\infty)
\end{equation*}
with
\begin{equation*}
\lim_{t \to 0^+} u(x,t) = |u_0(x)| \ \ \ \ {\rm a.e.}  \ \ {\rm in} \ \ \T^d.
\end{equation*}

\subsection{Maximal operators on the sphere} The set of techniques presented here allows us to address similar problems on other manifolds. We exemplify this by considering here the Poisson maximal operator and the heat flow maximal operator on the sphere $\S^{d}$. 

\subsubsection{Poisson maximal operator} Let $u_0\in L^p(\S^{d})$ with $1\leq p \leq \infty$. For $\omega \in \S^d$ and $0\leq \rho < 1$, let $u(\omega,\rho)=u(\rho\omega)$ be the function defined on the unit $(d+1)$-dimensional open ball $B_1 \subset \R^{d+1}$ as 
\begin{equation}\label{u_case_sphere}
u(\omega,\rho)= \int_{\S^{d}}\mc{P}(\omega,\eta, \rho)\,|u_0(\eta)|\,\d\sigma(\eta)\,,
\end{equation}
where $\mc{P}(\omega,\eta,\rho)$ is the Poisson kernel defined for $\omega, \eta \in \S^{d}$ by
\begin{equation*}\label{Poisson_kernel_sphere}
\mc{P}(\omega,\eta,\rho)=\frac{1-\rho^2}{ \sigma_{d}\,|\rho \omega-\eta|^d} = \frac{1-\rho^2}{ \sigma_{d} \,(\rho^2 - 2\rho\, \omega\cdot \eta + 1)^{d/2}}\,,
\end{equation*} 
with $\sigma_{d}$ being the surface area of $\S^{d}$.  In this case, we know that $u \in C^\infty(B_1)$ and it solves the Dirichlet problem
\begin{equation}\label{Diric}
\left\{\begin{array}{ll}
 \Delta u = 0 & {\rm in} \ B_1\,; \\
 \displaystyle\lim_{\rho \rightarrow 1}{u(\omega,\rho)}=|u_0(\omega)| & \mathrm{for~a.e.}~\omega\in \S^{d}.
\end{array}\right.
\end{equation}
From \cite[Chapter II, Theorem 2.3.6]{Yuan} we know that for each $0\leq\rho<1$ we have $u(\omega,\rho) \leq \M u_0(\omega)$, where $\M$ denotes de Hardy-Littlewood maximal operator on the sphere $\S^d$ (taken with respect to geodesic balls). Hence, we can define
\begin{equation}\label{maxpois}
u^*(\omega)=\sup_{0\leq\rho<1} u(\omega,\rho)
\end{equation}
and we know that $u_0 \mapsto u^*$ is bounded on $L^p(\S^d)$ for $1 < p \leq \infty$ (see \cite[Chapter II, Corollary 2.3.4]{Yuan}). Moreover, with an argument similar to \cite[proof of Theorem 1.4]{Ki}, using \eqref{Eq_crucial_Kin_sphere} and \eqref{Kin_sphere} below to explore the convolution structure of the sphere at the gradient level, one can show that $u_0\mapsto u^*$ is a bounded operator on $W^{1,p}(\S^d)$ for $1< p \leq \infty$, with $\|\nabla u^*\|_{L^p(\S^d)} \leq C\, \|\nabla u_0\|_{L^p(\S^d)}$ for some $C>1$.

\subsubsection{Heat flow maximal operator} Let $u_0\in L^p(\S^{d})$ with $1\leq p \leq \infty$. For $\omega \in \S^d$ and $t \in (0,\infty)$ let $u(\omega,t)$ be the function given by
\begin{equation}\label{sphere_heat_flow}
u(\omega,t)=\int_{\S^{d}}\mc{K}(\omega,\eta,t)\,|u_0(\eta)|\,\d\sigma(\eta)\,,
\end{equation}
where $\mc{K}(\omega,\eta,t)$ is the heat kernel on $\S^d$. Letting $\big\{Y_n^{\ell}\big\}$, $\ell = 1,2,\ldots, {\rm dim}\,{\mc H}_n^{d+1}$, be an orthonormal basis of the space ${\mc H}_n^{d+1}$ of spherical harmonics of degree $n$ in the sphere $\S^d$ (these are eigenvectors of the spherical Laplacian), we can write an explicit expression for this kernel as follows (see \cite[Lemma 1.2.3, Theorem 1.2.6 and Eq. 7.5.5]{Yuan}) 
\begin{equation*}
\mc{K}(\omega,\eta,t) = \sum_{n=0}^{\infty} e^{-t n (n+d -1)} \sum_{\ell=1}^{{\rm dim}\,{\mc H}_n^{d+1}} Y_n^{\ell}(\omega)Y_n^{\ell}(\eta)= \sum_{n=0}^{\infty} e^{-t n (n+d -1)}\frac{(n +\lambda)}{\lambda} \,C_n^{\lambda}(\omega \cdot \eta),
\end{equation*}
where $\lambda = \frac{d-1}{2}$ and $t \mapsto C^{\beta}_n(t)$, for $\beta >0$, are the {\it Gegenbauer polynomials} ({\it or ultraspherical polynomials}) defined in terms of the generating function 
\begin{equation*}
(1 - 2rt + r^2)^{-\beta} = \sum_{n=0}^{\infty} C^{\beta}_n(t)\, r^n.
\end{equation*}
As discussed in \cite[Chapter III, Section 2]{SY}, the kernel $\mc{K}$ verifies the following properties: 
\smallskip
\begin{enumerate}
\item[(P1)] $\mc{K}: \S^{d}\times \S^{d}\times (0,\infty) \to \R$ is a nonnegative smooth function that verifies $\displaystyle \partial_t\mc{K} - \Delta_{\omega} \mc{K} = 0$, where $\Delta_{\omega}$ denotes the Laplace-Beltrami operator with respect to the variable $\omega$.

\smallskip

\item[(P2)] $\mc{K}(\omega,\eta,t)=\mc{K}(\nu,t)$, where $\nu=d(\omega,\eta) = \arccos(\eta \cdot \omega)$ is the geodesic distance between $\omega$ and $\eta$. Moreover, we also have $\pd{\mc{K}}{\nu} < 0$, which means that $\mc{K}$ is radially decreasing in the spherical sense.

\smallskip

\item[(P3)] (Approximate identity) For each $t>0$ and $\omega \in \S^d$ we have
$$\displaystyle \int_{\S^{d}}\mc{K}(\omega,\eta,t)\,\d\sigma(\eta)=1,$$
and the function $u(\omega,t)$ defined in \eqref{sphere_heat_flow} converges pointwise a.e. to $|u_0|$ as $t \to 0$ (if $u_0\in C(\S^{d})$ the convergence is uniform).
\end{enumerate}
\smallskip
It then follows from (P1) and (P3) that $u(\omega,t)$ defined in \eqref{sphere_heat_flow} solves the heat equation
\begin{equation*}\label{Diric_heat}
\left\{\begin{array}{ll}
\partial_tu - \Delta u = 0 & {\rm in} \ \S^d \times (0,\infty)\,; \\
 \displaystyle\lim_{t \rightarrow 0^+}{u(\omega,t)}=|u_0(\omega)| & \mathrm{for~a.e.}~\omega\in \S^{d}.
\end{array}\right.
\end{equation*}
From (P2) and (P3) it follows from \cite[Chapter II, Theorem 2.3.6]{Yuan} that $u(\omega,t) \leq \mathcal{M}u_0(\omega)$, for each $t>0$. This allows us to define  
\begin{equation}\label{maxheats}
u^*(\omega)=\sup_{t>0}u(\omega,t)\,,
\end{equation}
and we see that $u_0 \mapsto u^*$ is bounded on $L^p(\S^d)$ for $1 < p \leq \infty$. As in the case of the Poisson maximal operator on $\S^d$ (or any maximal operator on the sphere associated to a smooth convolution kernel depending only on the inner product $\omega\cdot \eta$), using \eqref{Eq_crucial_Kin_sphere} below and \cite[proof of Theorem 1.4]{Ki}, one can show that $u_0\mapsto u^*$ is bounded on $W^{1,p}(\S^d)$ for $1< p \leq \infty$, with $\|\nabla u^*\|_{L^p(\S^d)} \leq C\, \|\nabla u_0\|_{L^p(\S^d)}$ for some $C>1$.

\subsubsection{Variation-diminishing property} Our next result establishes the variation-diminishing property for these maximal operators on the sphere $\S^d$.

\begin{theorem}\label{Thm3}
Let $u^*$ be the maximal function defined in \eqref{maxpois} or \eqref{maxheats}. The following propositions hold.
\smallskip
\begin{itemize}
\item[(i)] Let $1 < p \leq \infty$ and $u_0 \in W^{1,p}(\S^1)$. Then $u^{*} \in W^{1,p}(\S^1)$ and 
\begin{equation*}
\|(u^*)'\|_{L^p(\S^1)} \leq \|u_0'\|_{L^p(\S^1)}.
\end{equation*}

\vspace{0.15cm}

\item[(ii)] Let $u_0 \in W^{1,1}(\S^1)$. Then $u^{*} \in L^{\infty}(\S^1)$ and has a weak derivative $(u^*)'$ that satisfies 
\begin{equation*}
\|(u^*)'\|_{L^1(\S^1)} \leq \|u_0'\|_{L^1(\S^1)}.
\end{equation*}

\vspace{0.15cm}

\item[(iii)] Let $u_0$ be of bounded variation on $\S^1$. Then $u^*$ is of bounded variation on $\S^1$ and 
\begin{equation*}
V(u^*) \leq V(u_0).
\end{equation*}

\vspace{0.15cm}

\item[(iv)] Let $d>1$ and $u_0 \in W^{1,p}(\S^d)$, for $p =2$ or $p= \infty$. Then $u^{*} \in W^{1,p}(\S^d)$ and 
\begin{equation*}
\|\nabla u^*\|_{L^2(\S^d)} \leq \|\nabla u_0\|_{L^2(\S^d)}.
\end{equation*}

\end{itemize}
\end{theorem}

\noindent {\sc Remark}: Since $\S^1 \sim \T$, in the case of the heat flow maximal operator, parts (i), (ii) and (iii) of Theorem \ref{Thm3} have already been considered in Theorem \ref{Thm2}, and the novel part here is actually (iv).

\subsection{Non-tangential maximal operators} The last operator considered here is the classical non-tangential maximal operator associated to the Poisson kernel \eqref{Intro_Poisson}. For $\alpha \geq 0$ we consider 
\begin{equation}\label{def_max_function_cone}
u^*(x)=\sup_{\substack{t >0 \\ |y-x|\le \alpha t}} P(\cdot,t)*|u_0|(y).
\end{equation}
This operator is bounded on $L^p(\mathbb R^d)$ for $1 < p \leq \infty$ (see \cite[Chapter II, Equation (3.18)]{SW}). A modification of \cite[proof of Theorem 1.4]{Ki} (here one must discretize in time and in the set of possible directions) yields that this maximal operator is bounded on $W^{1,p}(\R)$ for $1<p\leq\infty$, with $\|\nabla u^*\|_{L^p(\R^d)} \leq C\, \|\nabla u_0\|_{L^p(\R^d)}$ for some $C>1$. Here we establish the variation-diminishing property of this operator in dimension $d=1$.

\begin{theorem}\label{Thm5}
Let $\alpha\geq0$ and let $u^*$ be the maximal function defined in \eqref{def_max_function_cone}. The following propositions hold.
\smallskip
\begin{itemize}
\item[(i)] Let $1 < p \leq \infty$ and $u_0 \in W^{1,p}(\R)$. Then $u^{*} \in W^{1,p}(\R)$ and 
\begin{equation*}
\|(u^*)'\|_{L^p(\R)} \leq \|u_0'\|_{L^p(\R)}.
\end{equation*}

\vspace{0.15cm}

\item[(ii)] Let $u_0 \in W^{1,1}(\R)$. Then $u^{*} \in L^{\infty}(\R)$ and has a weak derivative $(u^*)'$ that satisfies 
\begin{equation*}
\|(u^*)'\|_{L^1(\R)} \leq \|u_0'\|_{L^1(\R)}.
\end{equation*}

\vspace{0.15cm}

\item[(iii)] Let $u_0$ be of bounded variation on $\R$. Then $u^*$ is of bounded variation on $\R$ and 
\begin{equation*}
V(u^*) \leq V(u_0).
\end{equation*}

\end{itemize}
\end{theorem}

\subsection{A brief strategy outline} The proofs of Theorems \ref{Thm1} - \ref{Thm5} follow the same broad outline, each with their own technicalities. One component of the proof is to establish that it is sufficient to consider a Lipschitz continuous initial datum $u_0$. The second and crucial component of the proof is to establish that, for a Lipschitz continuous initial datum $u_0$, the maximal function is {\it subharmonic in the detachment set}. The steps leading to these results are divided in several auxiliary lemmas in the proofs of each theorem.

\smallskip

We remark that the subharmonicity property for the non-tangential maximal function \eqref{def_max_function_cone} in dimension $d>1$ is not true. We present a counterexample after the proof of Theorem \ref{Thm5}.

\section{Proof of Theorem \ref{Thm1}: Maximal operators and elliptic equations} \label{Section2}

\subsection{Preliminaries on the kernel}  \label{Section2.1}
Let $a>0$ and $b > 0$. We first observe that the function $\widehat{\varphi}_{a,b} (\cdot, t): \R^d \to \R$ defined in \eqref{Intro_Def_phi_hat_a_b} belongs to the Schwartz class for each $t>0$. Moreover, the function $g :[0,\infty) \to \R^+$ defined by 
$$ \widehat{\varphi}_{a,b} (\xi, t) =: g(|\xi|^2)$$
is {\it completely monotone}, in the sense that it verifies $(-1)^n g^{(n)}(s) \geq 0$ for $s>0$ and $n =0,1,2,\ldots$ and $g(0^+) = g(0)$. We may hence invoke a classical result of Schoenberg \cite[Theorems 2 and 3]{Scho} to conclude that there exists a finite nonnegative measure $\mu_{a,b,t}$ on $[0,\infty)$ such that 
\begin{equation*}
 \widehat{\varphi}_{a,b} (\xi, t) = \int_0^{\infty} e^{-\pi \lambda |\xi|^2} \,\d\mu_{a,b,t}(\lambda).  \end{equation*}
An application of Fubini's theorem gives us
\begin{equation}\label{Schoenberg_app}
{\varphi}_{a,b} (x, t) = \int_{\R^d} \widehat{\varphi}_{a,b}(\xi, t) \,e^{2 \pi i x \cdot \xi}\,\d\xi = \int_0^{\infty} \lambda^{-\frac{d}{2}}\, e^{-\frac{\pi |x|^2}{\lambda}} \,\d\mu_{a,b,t}(\lambda).
\end{equation}
In particular, \eqref{Schoenberg_app} implies that $\varphi_{a,b}(\cdot,t):\R^d \to \R$ is nonnegative and radial decreasing. It is convenient to record the explicit form of $\mu_{a,b,t}$. Starting from the identity \cite[page 6]{SW}, for $\beta >0$, 
\begin{equation*}
e^{-\beta} = \frac{1}{\sqrt{\pi}} \int_{0}^{\infty} \frac{e^{-u}}{\sqrt{u}} \, e^{-\frac{\beta^2}{4u}}\,\du = \frac{1}{2\pi} \int_0^{\infty} e^{-\pi \sigma \beta^2} \, e^{-\frac{1}{4\pi\sigma}}\, \sigma^{-\frac{3}{2}}\,\d\sigma,
\end{equation*}
we make $\beta = \frac{t}{2a} \left(b^2 + 16 a \pi^2 |\xi|^2\right)^{1/2}$ to obtain
\begin{equation}\label{Explicit_mu_a_b_t}
\d\mu_{a,b,t}(\lambda)= \left(e^{\frac{tb}{2a}}\,\frac{t}{\sqrt{a}} \,e^{-\frac{\lambda b^2}{16 \pi a}}\, e^{-\frac{\pi t^2}{a\lambda}}\, \,\lambda^{-\frac{3}{2}}\right)\,\d \lambda.
\end{equation}

\smallskip

From \eqref{Schoenberg_app}, \eqref{Explicit_mu_a_b_t} and dominated convergence we see that, for a fixed $x \neq 0$,
\begin{equation*}
\lim_{t \to 0^+} {\varphi}_{a,b} (x, t) = 0\,, 
\end{equation*}
and, for a fixed $\delta >0$,
\begin{equation}\label{Cond_2_int_conv_ap_id}
\lim_{t \to 0^+} \int_{|x| \geq \delta} {\varphi}_{a,b} (x, t) \,\dx = 0.
\end{equation}
For $f \in L^p(\R^d)$ with $1 \leq p < \infty$, it follows from  \eqref{Intro_eq_integral_1} and \eqref{Cond_2_int_conv_ap_id} that
\begin{equation}\label{L_p_conv_kernel_a_b}
\lim_{t\to 0^+} \|{\varphi}_{a,b} (\cdot, t) *f - f\|_{L^p(\R^d)} = 0.
\end{equation}
The additional fact that ${\varphi}_{a,b} (\cdot, t)$ is radial decreasing for each $t>0$ implies the pointwise convergence  
\begin{equation}\label{Point_conv_kernel_a_b}
\lim_{t\to 0^+} {\varphi}_{a,b} (\cdot, t) *f(x) = f(x) \ \ \ \ {\rm for \ \ a.e.} \ \ x \in \R^d.
\end{equation}
In \eqref{Point_conv_kernel_a_b} we may allow $f \in L^p(\R^d)$ with $1 \leq p \leq \infty$ and the convergence happens at every point in the Lebesgue set of $f$. The proofs of \eqref{L_p_conv_kernel_a_b} and \eqref{Point_conv_kernel_a_b} follow along the same lines of the proofs of \cite[Chapter I, Theorems 1.18 and 1.25]{SW} and we omit the details.

\smallskip

From \eqref{Schoenberg_app} and \eqref{Explicit_mu_a_b_t} we see that ${\varphi}_{a,b} \in C^{\infty}(\R^d \times (0,\infty))$. Moreover its decay is strong enough to assure that, if the initial datum $u_0 \in L^p(\R^d)$ for some $1 \leq p \leq \infty$, then $u(x,t) = {\varphi}_{a,b} (\cdot, t)*u_0(x) \in C^{\infty}(\R^d \times (0,\infty))$, with $D^{\alpha}u(x,t) = (D^{\alpha}{\varphi}_{a,b} (\cdot, t))*u_0(x)$ for any multi-index $\alpha \in (\Z^+)^{d+1}$. Finally, observe that $u(x,t)$ solves the partial differential equation
\begin{equation}\label{Ell_Eq_subhamonicity}
au_{tt} - bu_t + \Delta u = 0 \ \ \  {\rm in}  \ \ \  \R^d \times (0,\infty).
\end{equation}
This follows since the kernel $\varphi(x,t)$ solves the same equation, a fact that can be verified by differentiating under the integral sign the leftmost identity in \eqref{Schoenberg_app}. We also remark that if  $u_0 \in C(\R^d) \cap L^p(\R^d)$ for some $1\leq p < \infty$, or if $u_0$ is bounded and Lipschitz continuous, then the function $u(x,t)$ is continuous up to the boundary $\R^d \times \{t=0\}$ (this follows from \eqref{Point_conv_kernel_a_b} and \eqref{Imp_chain_cont_local} below).

\subsection{Auxiliary lemmas} In order to prove Theorem \ref{Thm1}, we may assume without loss of generality that $u_0 \geq 0$. In fact, if $u_0 \in W^{1,p}(\R^d)$ we have $|u_0| \in W^{1,p}(\R^d)$ and $|\nabla |u_0|| = |\nabla u_0|$ a.e. if $u_0$ is real-valued (in the general case of $u_0$ complex-valued we have $|\nabla |u_0|| \leq |\nabla u_0|$ a.e), and if $u_0$ is of bounded variation on $\R$ we have $V(|u_0|) \leq V(u_0)$. We adopt such assumption throughout the rest of this section.

\smallskip

The cases when $a=0$ (heat kernel) or $b=0$ (Poisson kernel) were already considered in \cite[Theorems 1 and 2]{CS}, so we focus in the remaining case $a>0$, $b>0$ \footnote{By appropriate dilations in the space variable $x$ and the time variable $t$, we could assume that $a=b=1$. However, this reduction is mostly aesthetical and offers no major technical simplification.}. We start with some auxiliary lemmas, following the strategy outlined in \cite{CS}. Throughout this section we write
$${\rm Lip}(u) = \sup_{\substack{x,y \in \R^d \\ x\neq y}} \frac{|u(x) - u(y)|}{|x-y|}$$
for the Lipschitz constant of a function $u:\R^d \to \R$. Let $B_r(x) \subset \R^d$ denote the open ball of radius $r$ and center $x$, and let $\overline{B_r(x)}$ denote the corresponding closed ball. When $x =0$ we shall simply write $B_r$.

\begin{lemma}[Continuity] \label{lem_continuity}
Let $a,b>0$ and $u^*$ be the maximal function defined in \eqref{def_max_function_a_b}.
\begin{itemize}
\item[(i)] If $u_0 \in C(\R^d) \cap L^p(\R^d)$, for some $1\leq p < \infty$, then $u^* \in C(\R^d)$. 
\smallskip
\item[(ii)] If $u_0$ is bounded and Lipschitz continuous then $u^*$ is bounded and Lipschitz continuous with ${\rm Lip}(u^*) \leq {\rm Lip}(u_0)$.
\end{itemize}
\end{lemma}
\begin{proof} Let us denote $\tau_hu_0:= u_0(x-h)$. Given $x \in \R^d$, we can choose $\delta >0$ such that
\begin{align}\label{Imp_chain_cont_local}
\begin{split}
& |\tau_h u_0 - u_0|*\varphi_{a,b}(\cdot,t)(x) \\
&= \int_{|y| < 1} |\tau_h u_0 - u_0|(x-y)\,\varphi_{a,b}(y,t)\,\dy + \int_{|y| \geq 1} |\tau_h u_0 - u_0|(x-y)\,\varphi_{a,b}(y,t)\,\dy\\
& \leq \sup_{w \in B_1(x)}|\tau_h u_0 - u_0|(w) + \|\tau_h u_0 - u_0\|_p\,\|\chi_{\{|\cdot|\geq 1\}} \,\varphi_{a,b}(\cdot,t)\|_{p'}\\
& < \varepsilon
\end{split}
\end{align}
whenever $|h| < \delta$, for all $t >0$. Above we have used the fact that $\|\chi_{\{|\cdot|\geq 1\}}\, \varphi_{a,b}(\cdot,t)\|_{p'}$ is uniformly bounded. Using the sublinearity, we then arrive at 
\begin{equation*}
\big|\tau_h u^*(x) - u^*(x)\big| \leq (\tau_h u_0 - u_0)^*(x) \leq \varepsilon
\end{equation*}
for $|h| < \delta$, which shows that $u^*$ is continuous at the point $x$.

\smallskip

\noindent (ii) Observe that for each $t>0$ the function $u(x,t) = \varphi_{a,b}(\cdot,t) * u_0(x)$ is bounded by $\|u_0\|_{\infty}$ and Lipschitz continuous with ${\rm Lip}(u(\cdot,t)) \leq {\rm Lip}(u_0)$. The result then follows since we are taking a pointwise supremum of uniformly bounded and Lipschitz functions.
\end{proof}

\begin{lemma}[Behaviour at large times]\label{lem_boundedness_large_times} Let $a,b>0$ and $u(x,t) = \varphi_{a,b}(\cdot, t)*u_0(x)$.
\begin{itemize}
\item[(i)] If $u_0 \in L^p(\R^d)$ for some $1\leq p < \infty$, then for a given $\varepsilon >0$ there exists a time $t_{\varepsilon} < \infty$ such that $\|u(\cdot, t)\|_{\infty} < \varepsilon$ for all $t > t_{\varepsilon}$.
\smallskip
\item[(ii)] If $u_0$ is bounded and if $r>0$ and $\varepsilon >0$ are given, then there exists a time $t_{r,\varepsilon} <\infty$ such that $ |u(x, t) -  u(y,t)| < \varepsilon$ for all $x,y \in B_r$  and $t > t_{r,\varepsilon}$.
\end{itemize}
\end{lemma}
\begin{proof} (i) The first statement follows from H\"{o}lder's inequality
$$\|u(\cdot, t)\|_{\infty} \leq \|u_0\|_p\,  \|\varphi_{a,b}(\cdot, t)\|_{p'}$$
and the fact that $\|\varphi_{a,b}(\cdot, t)\|_{p'} \to 0$ as $t \to \infty$. The latter follows from the estimate 
$$\|\varphi_{a,b}(\cdot, t)\|_{p'} \leq \|\varphi_{a,b}(\cdot, t)\|_{\infty}^{\frac{p' -1}{p'}}\  \|\varphi_{a,b}(\cdot, t)\|_{1}^{\frac{1}{p'}}\,,$$ 
observing that $\|\varphi_{a,b}(\cdot, t)\|_{1} = 1$ and $\|\varphi_{a,b}(\cdot, t)\|_{\infty} \to 0$ as $t \to \infty$ by the leftmost identity in \eqref{Schoenberg_app} and dominated convergence.

\smallskip

\noindent (ii) Since $\varphi_{a,b}(\cdot,t)$ is in the Schwartz class, for every index $k\in\{1,\ldots,d\}$ we have $$\frac{\partial u}{\partial x_k}(x,t)=\frac{\partial\varphi_{a,b}}{\partial x_k}(\cdot,t)*u_0(x).$$ This implies that $u(\cdot,t)$ is a Lipschitz function with constant bounded by $\|u_0\|_\infty\sum_{k=1}^d\big\|\frac{\partial\varphi_{a,b}}{\partial x_k}(\cdot,t)\big\|_1$. By \eqref{Schoenberg_app}, \eqref{Explicit_mu_a_b_t} and Fubini's theorem,
\begin{align*}
\left\|\frac{\partial\varphi_{a,b}}{\partial x_k}(\cdot,t)\right\|_1&=\left(\int_{\mathbb R^d}2\pi|x_k|e^{-\pi|x|^2}\d x\right)\left(\int_0^\infty\lambda^{-1/2}\, \d\mu_{a,b,t}(\lambda)\right)\\
&=\left(\int_{\mathbb R^d}2\pi|x_k|e^{-\pi|x|^2}\d x\right)\left(\int_0^\infty\frac{t}{\sqrt{a}\lambda^2}e^{-\frac{\lambda}{16\pi a}\left(b-\frac{4\pi t}{\lambda}\right)^2}\d\lambda\right).
\end{align*}
Setting $\lambda=t\nu$ and applying dominated convergence, one concludes that the second factor converges to $0$ as $t\to\infty$. The result plainly follows from this.

\end{proof}

We now start to explore the qualitative properties of the underlying elliptic equation \eqref{Ell_Eq_subhamonicity}. We say that a continuous function $f$ is {\it subharmonic} in an open set $A \subset \R^d$ if, for every $x \in A$, and every ball $\overline{B_r(x)} \subset A$ we have
$$f(x) \leq \frac{1}{\sigma_{d-1}}\int_{\S^{d-1}} f(x +r\xi)\,{\rm d}\sigma(\xi),$$
where $\sigma_{d-1}$ denotes the surface area of the unit sphere $\S^{d-1}$, and ${\rm d}\sigma$ denotes its surface measure. 

\begin{lemma}[Subharmonicity]\label{lem_subharmonicity} Let $a,b>0$ and $u^*$ be the maximal function defined in \eqref{def_max_function_a_b}. Let $u_0 \in C(\R^d) \cap L^p(\R^d)$ for some $1\leq p < \infty$ or $u_0$ be bounded and Lipschitz continuous. Then $u^*$ is subharmonic in the open set $A = \{x \in \R^d; \,u^*(x) > u_0(x)\}$.
\end{lemma}

\begin{proof}
From \eqref{Point_conv_kernel_a_b} we have $u^*(x) \geq u_0(x)$ for all $x\in \R^d$. From Lemma \ref{lem_continuity} we observe that $u^*$ is a continuous function and hence the set $A$ is indeed open. Let $x_0 \in A$ and $\overline{B_r(x_0)} \subset A$. Let $h:\overline{B_r(x_0)} \to \R$ be the solution of the Dirichlet boundary value problem
\begin{equation*}
\left\{
\begin{array}{rl}
\Delta h = 0& \ {\rm in} \ B_r(x_0);\\
h = u^*& \ {\rm in} \ \partial B_r(x_0).
\end{array}
\right.
\end{equation*}
Note that the auxiliary function $v(x,t) = u(x,t)-h(x)$ solves the equation 
\begin{equation*}
av_{tt} - bv_t + \Delta v = 0 \ \ \  {\rm in}  \ \ \  \ B_r(x_0)\times (0,\infty)
\end{equation*}
and it is continuous in $\overline{B_r(x_0)}\times [0,\infty)$, with $v(x,0) = u_0(x) - h(x)$. Let $y_0 \in \overline{B_r(x_0)}$ be such that $M =  \max_{x \in \overline{B_r(x_0)}} v(x,0) = v(y_0,0)$. We claim that $M \leq 0$. 

\smallskip

Assume that $M > 0$. Note that $v(x,t) \leq 0$ for every $x \in \partial B_r(x_0)$ and every $t>0$. This implies that $y_0 \in B_r(x_0)$. By the maximum principle, observe that $h \geq 0$ in $\overline{B_r(x_0)}$ and let $x_1 \in \partial B_r(x_0)$ be such that $\min_{x \in \overline{B_r(x_0)}} h(x) = h(x_1)$. Given $\varepsilon >0$, from Lemma \ref{lem_boundedness_large_times} we may find a time $t_0$ such that $|u(x,t_1) - u(y,t_1)| \leq \varepsilon$ for all $x,y \in \overline{B_r(x_0)}$ and $t_1 >t_0$. In particular, for any $x \in \overline{B_r(x_0)}$, we have
\begin{equation*}
v(x,t_1) \leq v(x,t_1) - v(x_1,t_1)  = u(x,t_1) - u(x_1,t_1) - (h(x) - h(x_1)) \leq u(x,t_1) - u(x_1,t_1) \leq \varepsilon\,,
\end{equation*}
for $t_1 >t_0$. If we take $\varepsilon <M$, the maximum principle applied to the cylinder $\Gamma = \overline{B_r(x_0)} \times [0,t_1]$ with $t_1 > t_0$ gives us
$$v(y_0,t) \leq v(y_0,0) = M$$
for all $0\leq t \leq t_1$. This plainly implies that $u(y_0, t) \leq u_0(y_0)$ for all $0 \leq t \leq t_1$. Since $t_1$ is arbitrarily large, we obtain $u^*(y_0) = u_0(y_0)$, contradicting the fact that $y_0 \in A$. This proves our claim.

\smallskip

Once established that $M \leq 0$, given $\varepsilon >0$ we apply again the maximum principle to the cylinder $\Gamma = \overline{B_r(x_0)} \times [0,t_1]$ with $t_1 > t_0$ as above to get $v(x_0,t) \leq \varepsilon$ for all $0 \leq t \leq t_1$. This implies that $u(x_0, t) \leq h(x_0) + \varepsilon$ for all $0 \leq t \leq t_1$, and since $t_1$ is arbitrarily large, we find that $u^*(x_0) \leq  h(x_0) + \varepsilon$. Since $\varepsilon >0$ is arbitrarily small, we conclude that
$$u^*(x_0) \leq h(x_0) =  \frac{1}{\sigma_{d-1}}\int_{\S^{d-1}} h(x_0 +r\xi)\,{\rm d}\sigma(\xi)= \frac{1}{\sigma_{d-1}}\int_{\S^{d-1}} u^*(x_0 +r\xi)\,{\rm d}\sigma(\xi)\,,$$
by the mean value property of the harmonic function $h$. This concludes the proof.
\end{proof}

The next lemma is a general result of independent interest. We shall use it in the proof of Theorem \ref{Thm1} for the case $p=2$.

\begin{lemma}\label{lem5_Int_by_parts}
Let $f, g \in C(\R^d) \cap W^{1,2}(\R^d)$ be real-valued functions with $g$ Lipschitz. Suppose that $g \geq 0$ and that $f$ is subharmonic in the open set $J = \{x\in \R^d; \ g(x) >0\}$. Then
\begin{equation*}
\int_{\R^d} \nabla f(x)\,.\,\nabla g(x)\,\,\dx\leq 0.
\end{equation*}
\end{lemma}
 t
\begin{proof}
This is \cite[Lemma 9]{CS}. To say a few words about this proof, an integration by parts at a formal level
$$\int_{\R^d} \nabla f\,.\,\nabla g\,\,\dx = \int_{\R^d} (-\Delta f) \, g\, \dx$$
would imply the result. However, in principle, $\Delta f$ is not a well-defined function, and one must be a bit careful and argue via approximation by smoother functions.
\end{proof}

\begin{lemma}[Reduction to the Lipschitz case]\label{Red_Lip_case}
In order to prove parts {\rm (i)}, {\rm (iii)} and {\rm (iv)} of Theorem \ref{Thm1} it suffices to assume that the initial datum $u_0:\R^d \to \R^+$ is bounded and Lipschitz. 
\end{lemma}

\begin{proof} Parts (i) and (iv). For the case $p=\infty$, recall that any function $u_0 \in W^{1,\infty}(\R^d)$ can be modified in a set of measure zero to become bounded and Lipschitz continuous. 

\smallskip

If $1 < p < \infty$, for $\varepsilon >0$ we write $u_{\varepsilon} = \varphi_{a,b}(\cdot, \varepsilon)*u_0$. It is clear that $u_{\varepsilon}$ is bounded, Lipschitz continuous and belongs to $W^{1,p}(\R^d)$. Assuming that the result holds for such $u_{\varepsilon}$, we would have $u_{\varepsilon}^* \in W^{1,p}(\R^d)$ with
\begin{equation}\label{Red_pf_eq_1}
\|\nabla u_{\varepsilon}^*\|_p \leq \|\nabla u_{\varepsilon}\|_p.
\end{equation}
Note that 
\begin{equation}\label{Red_pf_eq_sem_pro}
u_{\varepsilon}^*(x) = \sup_{t >0} \varphi_{a,b}(\cdot, t)*u_{\varepsilon}(x)  = \sup_{t >\varepsilon } \varphi_{a,b}(\cdot, t)*u_0(x),
\end{equation}
due to the semigroup property \eqref{Intro_Def_phi_hat_a_b}. Recall that there exists a universal $C >1$ such that 
\begin{equation}\label{Red_pf_eq_2}
\|u_{\varepsilon}^*\|_p \leq C\,\|u_{\varepsilon}\|_p.
\end{equation}
From Young's inequality (and also Minkowski's inequality in the case of the gradients) we have 
\begin{equation}\label{Red_pf_eq_3}
\|u_{\varepsilon} \|_p \leq \|u_0\|_p\ \ \ \ {\rm and} \ \ \ \  \|\nabla u_{\varepsilon} \|_p \leq \|\nabla u_0\|_p.
\end{equation}
From \eqref{Red_pf_eq_1}, \eqref{Red_pf_eq_2} and \eqref{Red_pf_eq_3} we see that $u_{\varepsilon}^*$ is uniformly bounded in $W^{1,p}(\R^d)$. From \eqref{Red_pf_eq_sem_pro} we have $u_{\varepsilon}^* \to u^*$ pointwise as $\varepsilon \to 0$. Hence, by the weak compactness of the space $W^{1,p}(\R^d)$, we must have $u^* \in W^{1,p}(\R^d)$ and $u_{\varepsilon}^* \rightharpoonup u^*$ as $\varepsilon \to 0$. It then follows from the lower semicontinuity of the norm under weak limits, \eqref{Red_pf_eq_1} and \eqref{Red_pf_eq_3} that
\begin{equation*}
\|\nabla u^*\|_p \leq \liminf_{\varepsilon \to 0} \|\nabla u_{\varepsilon}^*\|_p  \leq  \liminf_{\varepsilon \to 0} \|\nabla u_{\varepsilon}\|_p \leq  \|\nabla u_0\|_p.
\end{equation*}

\smallskip

\noindent Part (iii). Let $u_0:\R \to \R^+$ be of bounded variation. For $\varepsilon >0$ write $u_{\varepsilon} = \varphi_{a,b}(\cdot, \varepsilon)*u_0$. Then $u_{\varepsilon} \in C^{\infty}(\R)$ is bounded and Lipschitz continuous, and it is easy to see that $V(u_{\varepsilon}) \leq V(u_0)$. Assume that the result holds for such $u_{\varepsilon}$, i.e. that $V(u_{\varepsilon}^*) \leq  V(u_{\varepsilon})$. For any partition $\mc{P} = \{x_0 < x_1 < \ldots < x_N\}$ we then have
\begin{align}\label{Prep_limit_var_1}
V_{\mc{P}}(u_{\varepsilon}^*) := \sum_{n=1}^{N} |u_{\varepsilon}^*(x_n) - u_{\varepsilon}^*(x_{n-1})| \leq V(u_{\varepsilon}) \leq V(u_0).
\end{align}
By \eqref{Red_pf_eq_sem_pro}, we recall that $u_{\varepsilon}^* \to u^*$ pointwise as $\varepsilon \to 0$. Passing this limit in \eqref{Prep_limit_var_1} yields
\begin{align*}
V_{\mc{P}}(u^*) := \sum_{n=1}^{N} |u^*(x_n) - u^*(x_{n-1})| \leq V(u_0).
\end{align*}
Since this holds for any partition $\mc{P}$, we conclude that $V(u^*) \leq V(u_0)$. This completes the proof.
\end{proof}

The next lemma will be used in the proof of part (i) of Theorem \ref{Thm1}.

\begin{lemma}\label{Lem7_Renan}
Let $[\alpha, \beta]$ be a compact interval. Let $f,g: [\alpha, \beta] \to \R$ be absolutely continuous functions with $g$ convex. If $f(\alpha) = g(\alpha)$, $f(\beta) = g(\beta)$ and $f(x) < g(x)$ for all $x \in (\alpha, \beta)$, then 
\begin{equation}\label{Renan_eq0}
\|g'\|_{L^p([\alpha,\beta])} \leq \|f'\|_{L^p([\alpha,\beta])}
\end{equation}
for any $1 \leq p \leq \infty$.
\end{lemma}
\begin{proof}
Let us consider the case $1 \leq p < \infty$. The case $p=\infty$ follows by a passage to the limit in \eqref{Renan_eq0}. Assume that the right-hand side of \eqref{Renan_eq0} is finite, otherwise there is nothing to prove. Let $X \subset (\alpha,\beta)$ be the set of points where $g$ is differentiable and choose a sequence $\{x_n\}_{n=1}^{\infty}$ of elements of $X$ that is dense in $(\alpha,\beta)$. For each $x_n$ consider the affine function $L_n(x) := g(x_n) + g'(x_n)(x - x_n)$. Note that $L_n(x) \leq g(x)$ for all $x \in [\alpha,\beta]$. We set $f_0 = f$ and define inductively $f_{n+1} = \max\{f_n, L_{n+1}\}$. It is clear that each $f_n$ is absolutely continuous. Let $U_n = \{ x \in (\alpha,\beta);\,  L_{n+1}(x) > f_n(x)\}$. Then
\begin{equation}\label{Renan_eq1}
\int_{[\alpha,\beta]}|f_{n+1}'(x)|^p\,\dx = \int_{[\alpha,\beta]\setminus U_n}|f_{n}'(x)|^p\,\dx + m(U_n) \,|g'(x_{n+1})|^p.
\end{equation}
By Jensen's inequality, in each connected component $I = (r,s)$ of $U_n$ we have 
\begin{align}\label{Renan_eq2}
\int_I |f_n'(x)|^p\,\dx \geq (s-r) \left( \frac{1}{s-r} \int_I  |f_n'(x)|\,\dx \right)^p \geq (s-r) \left| \frac{f_n(s) - f_n(r)}{s-r}\right|^p = (s-r) \, |g'(x_{n+1})|^p.
\end{align}
By \eqref{Renan_eq1} and \eqref{Renan_eq2} we conclude that 
\begin{equation}\label{Renan_eq3}
\|f_{n+1}'\|_{L^p([\alpha,\beta])} \leq \|f_{n}'\|_{L^p([\alpha,\beta])}.
\end{equation}

\smallskip

Let $x \in X$. For sufficiently large $N$, there are indices $j,k \in \{1,2,\ldots, N\}$ such that $x_j \leq x < x_k$. Take these indices such that $x_j$ is as large as possible and $x_k$ is as small as possible. Since $f(x) < g(x)$, for large values of $N$ we have $f(x) < L_{j}(x)$ and $f(x) < L_k(x)$. Therefore $f_N(x) = \max\{f(x), L_1(x), \ldots, L_N(x)\}$ is either equal to $L_j(x)$ or $L_k(x)$. In fact, the function $f_N$ is differentiable in $x$ with $f_N'(x) = g'(x_j)$ or $f_N'(x) = g'(x_k)$, except in the case where $g'(x_j) \neq g'(x_k)$ and $L_j(x) = L_k(x)$, which only happens in a countable set of points $Y$. Assuming that $x \notin Y$ and that $g': X \to \R$ is continuous at $x$ (this is a set of full measure in $(\alpha,\beta)$) we have $f_N'(x) \to g'(x)$ as $N\to \infty$. From \eqref{Renan_eq3} and Fatou's lemma we get
\begin{equation*}
\|g'\|_{L^p([\alpha,\beta])} \leq \liminf_{N\to \infty} \|f_N'\|_{L^p([\alpha,\beta])} \leq \|f'\|_{L^p([\alpha,\beta])}.
\end{equation*}
\end{proof}
 
\noindent {\sc Remark}: If $f,g: [\alpha, \infty) \to \R$ are absolutely continuous functions with $g$ convex, and $f(\alpha) = g(\alpha) \geq 0$, $\lim_{x\to \infty}f(x) = \lim_{x \to \infty} g(x) = 0$ and $f(x) < g(x)$ for all $x \in (\alpha, \infty)$, the same proof of Lemma \ref{Lem7_Renan} gives
\begin{equation*}
\|g'\|_{L^p([\alpha,\infty))} \leq \|f'\|_{L^p([\alpha,\infty))}
\end{equation*}
for any $1 \leq p < \infty$. Observe in \eqref{Renan_eq1} that either $g'(x_{n+1}) = 0$ or $U_n$ is bounded. The same remark applies to the analogous situation on the interval $(-\infty, \beta]$.

\subsection{Proof of Theorem \ref{Thm1}} We are now in position to prove the main result of this section. 

\subsubsection{Proof of part {\rm (i)}} We defer the case $p=\infty$ to part (iv). Let us consider here the case $1 < p < \infty$. From Lemma \ref{Red_Lip_case} we may assume that $u_0 \in L^p(\R)$ is bounded and Lipschitz continuous. Then, from Lemma \ref{lem_continuity}, we find that $u^*$ is Lipschitz continuous and the detachment set $A = \{ x \in \R; \ u^*(x) > u_0(x)\}$ is open. Let us write $A$ as a countable union of open intervals 
\begin{equation}\label{writing_A}
A = \bigcup_j I_j = \bigcup_j \,(\alpha_j, \beta_j).
\end{equation} 
We allow the possibility of having $\alpha_j = -\infty$ or $\beta_j = \infty$, but note that, if $u_0 \not\equiv 0$, we must have $u^*(x_0) = u_0(x_0)$ at a global maximum $x_0$ of $u_0$, hence $A \neq (-\infty, \infty)$. From Lemma \ref{lem_subharmonicity}, $u^*$ is subharmonic (hence convex) in each subinterval $I_j = (\alpha_j, \beta_j)$. Part (i) now follows from Lemma \ref{Lem7_Renan} (and the remark thereafter, since $u_0, u^* \in L^p(\R)$).

\subsubsection{Proof of part {\rm (ii)}} Recall that a function $u_0 \in W^{1,1}(\R)$ can be modified in a set of measure zero to become absolutely continuous. Then, from Lemma \ref{lem_continuity} we find that $u^*$ is continuous and the detachment set $A = \{ x \in \R; \ u^*(x) > u_0(x)\}$ is open. Let us decompose $A$ as in \eqref{writing_A}. From Lemma \ref{lem_subharmonicity}, $u^*$ is subharmonic (hence convex) in each subinterval $I_j = (\alpha_j, \beta_j)$. Hence $u^*$ is differentiable a.e. in $A$, with derivative denoted by $v$. It then follows from Lemma \ref{Lem7_Renan} (and the remark thereafter, since $u^* \in L^1_{weak}(\R)$) that for each interval $I_j$ we have
\begin{equation}\label{pf_thm1_bv_w-1-1_case}
\int_{I_j} |v(x)|\,\dx \leq \int_{I_j} |u_0'(x)|\,\dx\,,
\end{equation}
and since $u_0' \in L^1(\R)$ we find that $v \in L^1(A)$. 

\smallskip

We now claim that $u^*$ is weakly differentiable with $(u^*)' = \chi_A. v + \chi_{A^c}. u_0'$. In fact, if $\psi \in C^{\infty}_c(\R)$ we have
\begin{align*}
\int_{\R} u^*(x)\, \psi'(x)\,\dx & = \int_{A^c} u_0(x)\, \psi'(x)\,\dx + \sum_j \int_{I_j} u^*(x)\, \psi'(x)\,\dx\\
& =  \int_{A^c} u_0(x)\, \psi'(x)\,\dx + \sum_j \left( u_0(\beta_j)\psi(\beta_j) - u_0(\alpha_j)\psi(\alpha_j) -  \int_{I_j} v(x)\, \psi(x)\,\dx\right)\\
& =  \int_{A^c} u_0(x)\, \psi'(x)\,\dx + \sum_j \left( \int_{I_j} u_0(x)\, \psi'(x)\,\dx + \int_{I_j} u_0'(x)\, \psi(x)\,\dx - \int_{I_j} v(x)\, \psi(x)\,\dx\right)\\
& = -  \int_{A^c} u_0'(x)\, \psi(x)\,\dx -  \int_{A} v(x)\, \psi(x)\,\dx,
\end{align*}
as claimed. Finally, using \eqref{pf_thm1_bv_w-1-1_case} we arrive at
\begin{equation*}
\int_{\R} |(u^*)'(x)|\,\dx = \int_{A} |v(x)|\,\dx +  \int_{A^c} |u_0'(x)|\,\dx \leq \int_{\R} |u_0'(x)|\,\dx,
\end{equation*}
which concludes the proof of this part.

\subsubsection{Proof of part {\rm (iii)}} By Lemma \ref{Red_Lip_case} we may assume that $u_0:\R \to \R^+$ of bounded variation is also Lipschitz continuous. By Lemma \ref{lem_subharmonicity} the function $u^*$ is subharmonic (hence convex) in the detachment set $A = \{ x \in \R; \ u^*(x) > u_0(x)\}$. This plainly leads to $V(u^*) \leq  V(u_0)$, since the variation does not increase in each connected component of $A$. 

\subsubsection{Proof of part {\rm (iv)}} We include here the case $d=1$ as well. If $p= \infty$, a function $u_0 \in W^{1,\infty}(\R^d)$ can be modified on a set of measure zero to become Lipschitz continuous with ${\rm Lip}(u_0) \leq \|\nabla u_0\|_{\infty}$. From Lemma \ref{lem_continuity}, the function $u^*$ is also be bounded and Lipschitz continuous, with ${\rm Lip}(u^*) \leq {\rm Lip}(u_0)$, and the result follows, since in this case $u^* \in W^{1,\infty}(\R^d)$ with $\|\nabla u^*\|_{\infty} \leq {\rm Lip}(u^*)$. 
\smallskip

If $p=2$, from Lemma \ref{Red_Lip_case} it suffices to consider the case where $u_0 \in W^{1,2}(\R^d)$ is Lipschitz continuous. In this case, we have seen from the discussion in the introduction and from Lemma \ref{lem_continuity} that the maximal function $u^* \in W^{1,2}(\R^d)$ is also Lipschitz continuous. From Lemma \ref{lem_subharmonicity}, $u^*$ is subharmonic in the detachment set $A = \{x \in \R^d; \,u^*(x) > u_0(x)\}$ and we may apply Lemma \ref{lem5_Int_by_parts} with $f = u^*$ and $g = (u^* - u_0)$ to get
\begin{align*}
\|\nabla u_0\|_2^2 &= \int_{\R^d} |\nabla u_0|^2\,\dx = \int_{\R^d} |\nabla u^* - \nabla (u^* - u_0)|^2\,\dx\\
& = \int_{\R^d} |\nabla (u^*-u_0)|^2\,\dx - 2 \int_{\R^d} \nabla u^*\,.\,\nabla (u^*-u_0)\,\dx + \int_{\R^d} |\nabla u^*|^2\,\dx\\
& \geq \int_{\R^d} |\nabla u^*|^2\,\dx = \|\nabla u^*\|_2^2.
\end{align*}
This concludes the proof.

\section{Proof of Theorem \ref{Thm2}: Periodic analogues} \label{Section3}

\subsection{Auxiliary lemmas} We follow here the same strategy used in the proof of Theorem \ref{Thm1}. We may assume in what follows that the initial datum $u_0$ is nonnegative. We now have to consider the whole range $a,b \geq 0$ with $(a,b) \neq (0,0)$.

\begin{lemma}[Continuity - periodic version] \label{lem_continuity_per}
Let $a,b \geq 0$ with $(a,b) \neq (0,0)$ and $u^*$ be the maximal function defined in \eqref{max_func_per}. 
\begin{itemize}
\item[(i)] If $u_0 \in C(\T^d)$ then $u^* \in C(\T^d)$.
\item[(ii)] If $u_0$ is Lipschitz continuous then $u^*$ is Lipschitz continuous with ${\rm Lip}(u^*) \leq {\rm Lip}(u_0)$.
\end{itemize}
\end{lemma}
\begin{proof}
Part (i). If $u_0 \in C(\T^d)$ then $u_0$ is uniformly continuous in $\T^d$. Therefore, given $\varepsilon >0$, there exists $\delta$ such that $|u_0(x-h) - u_0 (x)| \leq \varepsilon$ whenever $|h| \leq \delta$. It follows that (recall that $\tau_hu_0:= u_0(x-h)$)
\begin{align*}
\begin{split}
& |\tau_h u_0 - u_0|*\Psi_{a,b}(\cdot,t)(x) = \int_{\T^d} |\tau_h u_0 - u_0|(x-y)\,\Psi_{a,b}(y,t)\,\dy < \varepsilon
\end{split}
\end{align*}
if $|h| \leq \delta$, for every $t >0$. Using the sublinearity, we then arrive at 
\begin{equation*}
\big|\tau_h u^*(x) - u^*(x)\big| \leq (\tau_h u_0 - u_0)^*(x) \leq \varepsilon
\end{equation*}
for $|h| < \delta$, which shows that $u^*$ is continuous at the point $x$.

\smallskip

\noindent Part (ii). It follows since ${\rm Lip}(u(\cdot,t)) \leq {\rm Lip}(u_0)$ for each $t>0$.
\end{proof}

\begin{lemma}[Behaviour at large times - periodic version]\label{lem_boundedness_large_times_periodic} Let $a,b \geq 0$ with $(a,b) \neq (0,0)$ and $u(x,t) = \Psi_{a,b}(\cdot, t)*u_0(x)$. If $u_0:\T^d \to \R^+$ is bounded and if $r>0$ and $\varepsilon >0$ are given, then there exists a time $t_{r,\varepsilon} <\infty$ such that $ |u(x, t) -  u(y,t)| < \varepsilon$ for all $x,y \in B_r$  and $t > t_{r,\varepsilon}$.
\end{lemma}
\begin{proof}
It follows from \eqref{Connecting_two_max_fun} and Lemma \ref{lem_boundedness_large_times} (ii). 
\end{proof}

\begin{lemma}[Subharmonicity]\label{lem_subharmonicity_per} Let $a,b \geq 0$ with $(a,b) \neq (0,0)$ and $u^*$ be the maximal function defined in \eqref{max_func_per}. If $u_0 \in C(\T^d)$ then $u^*$ is subharmonic in the open set $A = \{x \in \T^d; \,u^*(x) > u_0(x)\}$.
\end{lemma}
\begin{proof}
Note initially that, by Lemma \ref{lem_continuity_per}, the function $u^*$ is continuous and the set $A \subset \T^d$ is indeed open. Moreover we have $A \neq \T^d$, since $u^*(x) = u_0(x)$ at a global maximum $x$ of $u_0$. The rest of the proof is similar to the proof of Lemma \ref{lem_subharmonicity}, using the maximum principle for the heat equation in the case $a=0$.
\end{proof}

\begin{lemma}\label{lem5_Int_by_parts_per}
Let $f, g \in C(\T^d) \cap W^{1,2}(\T^d)$ with $g$ Lipschitz. Suppose that $g \geq 0$ and that $f$ is subharmonic in the open set $J = \{x\in \T^d; \ g(x) >0\}$. Then
\begin{equation*}
\int_{\T^d} \nabla f(x)\,.\,\nabla g(x)\,\,\dx\leq 0.
\end{equation*}
\end{lemma}
\begin{proof}
This follows as in \cite[Lemma 9]{CS}. We omit the details.
\end{proof}

\begin{lemma}[Reduction to the Lipschitz case - periodic version]\label{Red_Lip_case_per}
In order to prove parts {\rm (i)}, {\rm (iii)} and {\rm (iv)} of Theorem \ref{Thm2} it suffices to assume that the initial datum $u_0:\T^d \to \R^+$ is Lipschitz. 
\end{lemma}
\begin{proof}
This follows as in the proof of Lemma \ref{Red_Lip_case}.
\end{proof}

\subsection{Proof of Theorem \ref{Thm2}} Once we have established the lemmas of the previous subsection, together with Lemma \ref{Lem7_Renan}, the proof of Theorem \ref{Thm2} follows essentially as in the proof of Theorem \ref{Thm1}. We omit the details.

\section{Proof of Theorem \ref{Thm3}: Maximal operators on the sphere}

\subsection{Auxiliary lemmas} As before, we may assume that the initial datum $u_0$ is nonnegative.

\smallskip

In this section we denote by $B_r(\omega) \subset \S^{d}$ the geodesic ball of center $\omega$ and radius $r$, i.e. 
$$B_r(\omega) = \{ \eta \in \S^d;\,  d(\eta, \omega) = \arccos(\eta \cdot \omega) < r\}.$$
We say that a continuous function $f:\S^d\to\R$ is {\it subharmonic} in a relatively open set $A \subset \S^d$ if, for every $\omega \in A$, and every geodesic ball $\overline{B_r(\omega)} \subset A$ we have
$$f(\omega) \leq \frac{1}{\sigma(\partial B_r(\omega))}\int_{\partial B_r(\omega)} f(\eta)\,{\rm d}\sigma(\eta),$$
where $\sigma(\partial B_r(\omega))$ denotes the surface area of $\partial B_r(\omega)$, and ${\rm d}\sigma$ denotes its surface measure. Throughout this section we write
$${\rm Lip}(u) = \sup_{\substack{\omega,\eta \in \S^d \\ \omega\neq \eta}} \frac{|u(\omega) - u(\eta)|}{d(\omega, \eta)}$$
for the Lipschitz constant of a function $u:\S^d \to \R$.

\begin{lemma}[Continuity - spherical version] \label{lem_continuity_pois}
Let $u^*$ be the maximal function defined in \eqref{maxpois} or \eqref{maxheats}. 
\begin{itemize}
\item[(i)] If $u_0 \in C(\S^d)$ then $u^* \in C(\S^d)$.
\item[(ii)] If $u_0$ is Lipschitz continuous then $u^*$ is Lipschitz continuous with ${\rm Lip}(u^*) \leq {\rm Lip}(u_0)$.
\end{itemize}
\end{lemma}
\begin{proof}
(i) For the Poisson kernel this follows easily from the uniform continuity of $u$ defined in \eqref{u_case_sphere} in the unit ball $\overline{B_1} \subset \R^{d+1}$. For the heat kernel we use the fact that the function $u(\omega, t)$ defined in \eqref{sphere_heat_flow} converges uniformly to the average value $M = \frac{1}{\sigma_d} \int_{\S^d} u_0(\eta) \, \d\sigma(\eta)$ as $t \to \infty$, which implies that $u$ is uniformly continuous in $\S^d \times [0,\infty)$.

\smallskip

\noindent(ii) Let us consider the case of the Poisson kernel. The case of the heat kernel is analogous. Fix $0 < \rho < 1$ and consider two vectors $\omega_1$ and $\omega _2$ in $\S^d$. Let $E = {\rm span}\{\omega_1, \omega_2\}$ and F be the orthogonal complement of $E$ in $\R^{d+1}$. Let $T$ be an orthogonal transformation in $\R^{d+1}$ such $T|_E$ is a rotation with  $T\omega_1 = \omega_2$ and $T|_F = I$. It follows that for any $\eta \in \S^d$ we have $d(\eta, T\eta) \leq d(\omega_1, \omega_2)$. Using the fact that the Poisson kernel $\mc{P}(\omega,\eta, \rho)$ depends only on the inner product $\omega \cdot \eta$ (the same holds for the heat kernel) we have
\begin{align*}
|u(\omega_1,\rho) - u(\omega_2,\rho)| & = \left|\int_{\S^{d}}\mc{P}(\omega_1,\eta, \rho)\,u_0(\eta)\,\d\sigma(\eta) - \int_{\S^{d}}\mc{P}(\omega_2,\eta, \rho)\,u_0(\eta)\,\d\sigma(\eta)\right|\\
& = \left|\int_{\S^{d}}\mc{P}(\omega_1,\eta, \rho)\,u_0(\eta)\,\d\sigma(\eta) - \int_{\S^{d}}\mc{P}(T^{-1}\omega_2,\eta, \rho)\,u_0(T\eta)\,\d\sigma(\eta)\right|\\
& \leq \int_{\S^{d}}\mc{P}(\omega_1,\eta, \rho)\,\big|u_0(\eta) - u_0(T\eta)\big|\,\d\sigma(\eta) \\
& \leq \int_{\S^{d}}\mc{P}(\omega_1,\eta, \rho)\, {\rm Lip}(u_0) \, d(\eta, T\eta)\,\d\sigma(\eta)\\
& \leq {\rm Lip}(u_0)\, d(\omega_1, \omega_2).
\end{align*}
Hence ${\rm Lip}(u(\cdot, \rho)) \leq {\rm Lip}(u_0)$ and the pointwise supremum of Lipschitz functions with constants at most ${\rm Lip}(u_0)$ is also a Lipschitz function with constant at most ${\rm Lip}(u_0)$. 
\end{proof}

\begin{lemma}[Subharmonicity - spherical version] Let $u^*$ be the maximal function defined in \eqref{maxpois} or \eqref{maxheats}. If $u_0 \in C(\S^d)$ then $u^*$ is subharmonic in the open set $A = \{x \in \S^d; \,u^*(\omega) > u_0(\omega)\}$. 
\end{lemma}
\begin{proof} First we deal with the maximal function associated to the Poisson kernel in \eqref{maxpois}. By Lemma \ref{lem_continuity_pois}  we know that $u^*$ is continuous and the set $A$ is indeed open. Take $\omega_0\in A$ and consider a radius $r>0$ such that the closed geodesic ball
$\ov{B_r(\omega_0)}$ is contained in $A$. Let $h:\ov{B_r(\omega_0)}\rightarrow \R$ be the solution of the Dirichlet problem
$$ \left\{\begin{array}{rl}
 \Delta h = 0~& {\rm in }  \ \, B_r(\omega_0); \\
\ h = u^*~ & {\rm in } \   \, \partial B_r(\omega_0),
\end{array}\right.$$
where $\Delta=\Delta_{\S^{d}}$ is the Laplace-Beltrami operator with respect to the usual metric in $\S^{d}$. Since $u^*$ is continuous, the unique solution $h$ belongs to $C^2(B_r(\omega_0))\cap C(\ov{B_r(\omega_0)})$. We now define the function 
$$v(\omega,\rho)=u(\omega,\rho)-h(\omega),$$  
which is harmonic (now with respect to the Euclidean Laplacian) in the open set
$U=\{\rho\omega \in \R^{d}; \,\omega \in B_r(\omega_0),~0 < \rho < 1\}$. We claim that $v \leq 0$ in $U$. Assume that this is not the case and let
\begin{equation}\label{Em_add_0}
M = \sup_U v(\omega,\rho) > 0.
\end{equation}
Let $\omega_1 \in \partial B_r(\omega_0)$ (by the maximum principle) be such that 
\begin{equation}\label{Em_add_1}
\min_{\omega \in \ov{B_r(\omega_0)}} h(\omega) = h(\omega_1).
\end{equation}
Since $u$ is continuous in the unit Euclidean ball, let $\varepsilon>0$ be such that (recall that we identify $u(\omega,\rho) = u(\rho \omega)$) 
\begin{equation}\label{Em_add_2}
|u(\omega,\rho) - u(0)| \leq \frac{M}{2}
\end{equation}
for $0 \leq \rho \leq \varepsilon$. Therefore, for $0 < \rho \leq \varepsilon$, by \eqref{Em_add_1} and \eqref{Em_add_2} we have
\begin{align}\label{Em_add_3}
v(\omega,\rho)=u(\omega,\rho)-h(\omega) \leq \left(u(0) + \frac{M}{2}\right) - h(\omega_1) \leq \left(u^*(\omega_1) + \frac{M}{2}\right) - h(\omega_1) = \frac{M}{2}.
\end{align}
Let $U_{\varepsilon} = \{\rho\omega \in \R^{d}; \,\omega \in B_r(\omega_0),~\varepsilon < \rho < 1\}$. Note that $v$ is continuous up to the boundary of $U_{\varepsilon}$ and by \eqref{Em_add_0} and \eqref{Em_add_3} we have 
\begin{equation*}
M = \max_{\ov{U_\varepsilon}} v(\omega,\rho).
\end{equation*}
By the maximum principle, this maximum is attained at the boundary of $U_\varepsilon$. From \eqref{Em_add_3} we may rule out the set where $\rho = \varepsilon$. Since $h=u^*$ in $\partial B_r(\omega_0)$, we have $v \leq 0$ in the set $\{\rho\omega \in \R^{d}; \,\omega \in \partial B_r(\omega_0),~\varepsilon \leq \rho \leq 1\}$. Hence the maximum $M$ must be attained at a point $\eta \in B_r(\omega_0)$ (and $\rho =1$). It follows that 
\begin{equation*}
u(\eta, \rho)-h(\eta)\leq u_0(\eta)-h(\eta)
\end{equation*}
for every $0< \rho < 1$, which implies that $u^*(\eta)=u_0(\eta)$, a contradiction. This establishes our claim. 

\smallskip

It then follows that $u(\omega, \rho) \leq h(\omega)$ for any $\omega \in B_r(\omega_0)$ and $0< \rho <1$, and this yields $u^*\leq h$ in $B_r(\omega_0)$. Since this is true for any $\omega_0\in A$ and any $r>0$ such that 
$B_r(\omega_0)\subset A$, we conclude that $u^*$ is subharmonic in $A$.

\smallskip

The proof for the maximal operator associated to the heat kernel \eqref{maxheats} follows along the same lines (see the proof of \cite[Lemma 8]{CS}), using the maximum principle for the heat equation.
\end{proof}

\begin{lemma} \label{integ_partes_sphere}
Let $f, g \in C(\S^d) \cap W^{1,2}(\S^d)$ be real-valued functions. Suppose that $g \geq 0$ and that $f$ is subharmonic in the open set $J = \{\omega\in \S^d; \ g(\omega) >0\}$. Then
\begin{equation*}
\int_{\S^d}\nabla f(\omega)\cdot\nabla g(\omega)~\d\sigma(\omega)\leq 0.
\end{equation*}
\end{lemma}
\begin{proof} 
If both functions were smooth, the result would follow from integration by parts on the sphere (see \cite[Chapter I, Proposition 1.8.7]{Yuan}), since
\begin{equation*}
\int_{\S^d}\nabla f\cdot\nabla g~\d\sigma(\omega)=\int_{\S^d}(-\Delta f)~ g~\d\sigma(\omega)\leq 0
\end{equation*}
and $-\Delta f \leq 0$ in the set where $g>0$. To prove the result we approximate $f$ and $g$ by smooth functions in a suitable way.

\smallskip

Let $O(d+1)$ be the group of rotations of $\R^{d+1}$ and let $\mu$ be its Haar probability measure. We consider a family $\psi_\varepsilon$ of nonnegative $C^\infty$-functions in $O(d+1)$ 
supported in an $\varepsilon$-neighborhood of the identity transformation with 
$$\int_{O(d+1)}\psi_\varepsilon( R)\, \d\mu(R)=1\,,$$
and we ask for each $\varepsilon$ that $\psi_\varepsilon(S^tRS)=\psi_\varepsilon(R)$ for every $S \in O(d+1)$, i.e., that $\psi_\varepsilon$ is invariant under conjugation. To construct such $\psi_\varepsilon$, it is enough to consider a smooth function of the trace in $O(d+1)$ which is concentrated in the set where the trace is in a neighborhood of $d+1$. 
We now define $f_\varepsilon$ by
\begin{equation}\label{heps}
f_\varepsilon(\omega)=\int_{O(d+1)}f(R\omega)\,\psi_\varepsilon( R)\, \d\mu( R).
\end{equation}
We now observe the following facts:

\smallskip

\noindent 1. The function $f_\varepsilon \in C^\infty(\S^d)$. To see this we argue as follows. Let $e_1$ be the first canonical vector of $\R^{d+1}$ and define $F_{\varepsilon}: O(d+1) \to \R$ by 
\begin{align*}
F_{\varepsilon}(T) & = \int_{O(d+1)}f(RTe_1)\,\psi_\varepsilon( R)\, \d\mu( R)\\
& = \int_{O(d+1)}f(Re_1)\,\psi_\varepsilon( RT^{-1})\, \d\mu( R).
\end{align*}
Since $\psi_\varepsilon( RT^{-1})$ is smooth as a function of $R$ and $T$, the function $F_{\varepsilon}$ is also smooth. Then the equality $F_{\varepsilon} = f_\varepsilon(Te_1)$ and the fact that $T \mapsto Te_1$ is a smooth submersion from $O(d+1)$ to $\S^d$ imply that $f_\varepsilon$ is also smooth.

\smallskip

\noindent 2. The family $f_\varepsilon$ approximates $f$ in $W^{1,2}(\S^d)$ as $\varepsilon \to 0$. This can be verified directly from \eqref{heps}.

\smallskip

\noindent 3. The function $f_\varepsilon$ is subharmonic in the set $J_\varepsilon:=\{\omega\in J; \,d(\omega,\partial J)>\varepsilon\}$. In fact, using the invariance of geodesic spheres under rotations and Fubini's theorem we find, for $\omega \in J_\varepsilon$,
\begin{align*}
f_\varepsilon(\omega)& =\int_{O(d+1)}f(R\omega)\,\psi_\varepsilon( R)\, \d\mu( R)\\
& \leq \int_{O(d+1)}\left(\frac{1}{\sigma(\partial B_r(R\omega))}\int_{\partial B_r(R\omega)} f(\eta)\,{\rm d}\sigma(\eta)\right)\psi_\varepsilon( R)\, \d\mu( R)\\
& = \int_{O(d+1)}\left(\frac{1}{\sigma(\partial B_r(\omega))}\int_{\partial B_r(\omega)} f(R\zeta)\,{\rm d}\sigma(\zeta)\right)\psi_\varepsilon( R)\, \d\mu( R)\\
& = \frac{1}{\sigma(\partial B_r(\omega))}\int_{\partial B_r(\omega)} f_{\varepsilon}(\zeta)\,{\rm d}\sigma(\zeta).
\end{align*}
Since $f_\varepsilon$ is smooth, this implies that $(-\Delta f_\varepsilon)\leq 0$ in $J_\varepsilon$.

\smallskip

\noindent 4. This is more a remark and will not be strictly necessary for our proof. The function $f_\varepsilon$ can be given as a convolution with a kernel that depends on the inner product of the entries. In fact, by the co-area formula one gets
\begin{align*}\label{smoothms}
\begin{split}
f_\varepsilon(\omega)& =\int_{O(d+1)}f(R\omega)\,\psi_\varepsilon( R)\, \d\mu(R) \\
                     & =\int_{\S^d} f(\eta)\int_{\{R\omega=\eta\}} \psi_\varepsilon( R)\,\llbracket F_\omega(R)\rrbracket^{-1} ~\d \H^{d(d-1)/2}(R)~\d\sigma(\eta)\\
                     &=\int_{\S^d}f(\eta)\,\Psi_\varepsilon(\omega,\eta)\,\d\sigma(\eta),
\end{split}
\end{align*}
where $\llbracket F_\omega(R) \rrbracket$ is the Jacobian of the submersion $F_\omega(R)=R\omega$ (this is just a constant) and $\H^{d(d-1)/2}$ is the $[d(d-1)/2]$-dimensional Hausdorff measure of $(O(d+1), \d\mu)$. From the invariance of $\psi_\varepsilon$ by conjugation, it follows that $\Psi_\varepsilon(\omega, \eta)$ depends only on the inner product $\omega \cdot \eta$. The advantage of defining $f_\varepsilon$ as in \eqref{heps} is that we easily get the subharmonicity in $J_\varepsilon=\{\omega\in J; \, d(\omega,\partial J)>\varepsilon\}$ as shown in (3) above. In contrast to $\R^{d}$, there is no canonical way to move geodesic spheres that works in the same way as translation does in the Euclidean space, hence our choice to average over the whole group of rotations to arrive at this specific convolution kernel.

\smallskip

We now conclude the proof. Since $g$ is continuous, for each $\varepsilon>0$ there is a $\delta=\delta(\varepsilon)$, which goes to $0$ as $\varepsilon$ goes to $0$, such that $g(\omega) \leq \delta $ for each $\omega \in J\setminus J_\varepsilon$. We then consider the function $g_\delta=(g-\delta)_+$, i.e. the function that is $g-\delta$ when $g \geq \delta$ and $0$ otherwise. Then $g_\delta \to g$ in $W^{1,2}(\S^d)$ as $\delta \to 0$ and it follows that
\begin{equation}\label{fin_1_2}
\int_{\S^d}\nabla f_\varepsilon\cdot\nabla g_\delta~\d\sigma(\omega)\rightarrow \int_{\S^d}\nabla f\cdot\nabla g~\d\sigma(\omega).
\end{equation}
By integration by parts we have
\begin{align}
\begin{split}\label{intjeps}
\int_{\S^d}\nabla f_\varepsilon\cdot\nabla g_\delta~\d\sigma(\omega) &= \int_{\S^d}(-\Delta f_\varepsilon)~ g_{\delta}~\d\sigma(\omega) \\
                                                  &=\int_{J_\varepsilon}(-\Delta f_\varepsilon)~ g_\delta~\d\sigma(\omega)+  \int_{J\setminus J_\varepsilon}(-\Delta f_\varepsilon)~ g_\delta~\d\sigma(\omega)\\
                                                  &\leq 0.
\end{split}
\end{align}
The result follows from \eqref{fin_1_2} and \eqref{intjeps}.
\end{proof}

\begin{lemma}[Reduction to the continuous case - spherical version]
In order to prove parts {\rm (i)}, {\rm (iii)} and {\rm (iv)} of Theorem \ref{Thm3} it suffices to assume that the initial datum $u_0: \S^d \to \R^+$ is continuous. 
\end{lemma}

\begin{proof} We consider here the Poisson case and the heat flow case is analogous. For $0<r<1$ and $\omega\in\S^d$ let $u_r(\omega)=u(r\omega)$. It is clear that $u_r$ is a continuous function (in fact it is smooth) and that the solution of the Dirichlet problem \eqref{Diric}, with $u_r$ replacing $u_0$ as the boundary condition, is a suitable dilation of $u$ defined in \eqref{u_case_sphere}. Hence
$$u_r^*(\omega)=\sup_{0\leq\rho<r}u(\rho\omega),$$ 
which implies that $u_r^*\to u^*$ pointwise as $r\to1$.

\smallskip

For any $\omega,v\in\S^d$ such that $\omega\cdot v=0$, let $T$ be the linear transformation such that $T(\omega)=v$, $T(v)=-\omega$ and $T(\zeta)=0$ whenever $\zeta$ is orthogonal to $\omega$ and $v$. For $\lambda \in \R$ observe that $e^{\lambda T}$ is a rotation on $\R^{d+1}$ and hence
\begin{align*}
u_r(e^{\lambda T} \omega) & =  \int_{\S^{d}}\mc{P}(e^{\lambda T} \omega,\zeta, r)\,u_0(\zeta)\,\d\sigma(\zeta)\\
& = \int_{\S^{d}}\mc{P}(\omega,\eta, r)\,u_0(e^{\lambda T}\eta)\,\d\sigma(\eta).
\end{align*}
Differentiating both sides with respect to $\lambda$ and evaluating at $\lambda =0$ yields 
\begin{equation*}
\nabla u_r(\omega)\cdot v=\int_{\S^d}\mc{P}(\omega,\eta, r) \left(\nabla u_0(\eta)\cdot T(\eta)\right)\d\sigma(\eta).
\end{equation*}
We then observe that 
\begin{equation}\label{Eq_crucial_Kin_sphere}
|\nabla u_r(\omega)| \leq \int_{\S^d}\mc{P}(\omega,\eta, r) \,|\nabla u_0(\eta)|\, \d\sigma(\eta).
\end{equation}
It follows that 
\begin{equation}\label{Kin_sphere}
|\nabla u_r(\omega)| \leq |\nabla u_0|^*(\omega)
\end{equation}
and, by \eqref{Eq_crucial_Kin_sphere} and Jensen's inequality, we obtain
\begin{equation*}
\|\nabla u_r\|_{L^p(\S^d)} \leq \|\nabla u_0\|_{L^p(\S^d)}
\end{equation*}
for $1 \leq p \leq \infty$. The rest of the proof follows as in Lemma \ref{Red_Lip_case}.
\end{proof}

\subsection{Proof of Theorem \ref{Thm3}} Combining the lemmas of the previous subsection with Lemma \ref{Lem7_Renan}, the proof of Theorem \ref{Thm3} follows as in the proof of Theorem \ref{Thm1}.  We omit the details.

\section{Proof of Theorem \ref{Thm5}: Non-tangential maximal operators}

\subsection{Auxiliary lemmas} We keep the same strategy. The first step is still to note that the initial condition $u_0$ may be assumed to be nonnegative. In this section $u(x,t)=P(\cdot,t)*u_0(x)$ for $t>0$ and $u(x,0)=u_0(x)$. The function $u$ defined this way is harmonic in the open upper half-plane. We may restrict ourselves to the novel case $\alpha >0$.

\begin{lemma}[Continuity - non-tangential version] \label{lem_continuity_cone}
Let $\alpha >0$ and $u^*$ be the maximal function defined in \eqref{def_max_function_cone}.
\begin{itemize}
\item[(i)] If $u_0 \in C(\R) \cap L^p(\R)$, for some $1\leq p < \infty$, then $u^* \in C(\R)$. 
\smallskip
\item[(ii)] If $u_0$ is bounded and Lipschitz continuous then $u^*$ is bounded and Lipschitz continuous with ${\rm Lip}(u^*) \leq {\rm Lip}(u_0)$.
\end{itemize}
\end{lemma}
\begin{proof}
(i) From the hypothesis $u_0\in C(\mathbb R)\cap L^p(\mathbb R)$, we know that $u$ is continuous up to the boundary. By H\"{o}lder's inequality, $|u(x,t)|\le\|P(\cdot,t)\|_{p'}\|u_0\|_p$ and so $u(x,t)$ converges uniformly to zero as $t\to\infty$. These facts imply that $u^*\in C(\mathbb R)$.

\smallskip

\noindent (ii) For any $t>0$ and $y\in\mathbb R$, the function $x \mapsto u(x+y,t)$ is bounded by $\|u_0\|_{\infty}$ and is Lipschitz continuous with constant less than or equal to ${\rm Lip}(u_0)$. The claim follows since $u^*(x)$ is the supremum of these functions over all pairs $(t,y)$ such that $|y|\le \alpha t$.
\end{proof}

\begin{lemma}[Subharmonicity - non-tangential version] \label{Lem22} Let $\alpha>0$ and $u^*$ be the maximal function defined in \eqref{def_max_function_cone}. Let $u_0 \in C(\R) \cap L^p(\R)$ for some $1\leq p < \infty$ or $u_0$ be bounded and Lipschitz continuous. Then $u^*$ is subharmonic in the open set $A = \{x \in \R; \,u^*(x) > u_0(x)\}$.
\end{lemma}
\begin{proof}
The set $A$ is in fact open due to Lemma \ref{lem_continuity_cone}. 

\smallskip

\noindent {\it Step 1}. We first prove the following claim: for any $x_0\in A$ there exist arbitrarily small positive values of $\varepsilon$ such that 
\begin{equation}\label{claim_1_Renan}
u^*(x_0+\varepsilon)+u^*(x_0-\varepsilon)\ge 2u^*(x_0).
\end{equation}

\smallskip

\noindent {\it Case 1}. Assume that $u_0$ is bounded and Lipschitz continuous and that 
\begin{equation}\label{Sub_Renan_0}
d=u^*(x_0)-\sup_{\substack{t >0 \\ |y-x_0| = \alpha t}} u(y,t) > 0.
\end{equation} 
Since ${\rm Lip}(u(\cdot,t))\le{\rm Lip}(u_0)$ for any positive $t$, we have 
\begin{equation}\label{Sub_Renan_1}
u^*(x_0)=\sup_{\substack{t>0 \\ |y-x_0|\le \alpha t-\frac{d}{2{\rm Lip}(u_0)}}} u(y,t).
\end{equation} 
For $0<\varepsilon<\frac{d}{2{\rm Lip}(u_0)}$ the region over which we take the supremum in \eqref{Sub_Renan_1} is contained in the region $|y-(x_0+\varepsilon)|\le \alpha t$ and so $u^*(x_0+\varepsilon)\ge u^*(x_0)$. Similarly $u^*(x_0-\varepsilon)\ge u^*(x_0)$, and this establishes \eqref{claim_1_Renan}.

\smallskip

\noindent {\it Case 2}. Let us define two operators: $u_R^*(x) = \sup_{t >0} u(x + \alpha t,t)$ and $u_L^*(x) = \sup_{t >0} u(x - \alpha t,t)$. If \eqref{Sub_Renan_0} does not happen then 
\begin{equation}\label{Sub_Renan_2}
u^*(x_0) = \max\{ u_R^*(x_0), u_L^*(x_0)\}.
\end{equation}
This is certainly the case when $u_0 \in C(\R) \cap L^p(\R)$, since the function $u(x,t)$ converges to zero uniformly as $t\to\infty$ and \eqref{Sub_Renan_2} follows by the maximum principle. Let us assume without loss of generality that $u^*(x_0) = u_R^*(x_0)$. 

\smallskip

Let $\theta=\arctan \alpha$ and let $T:\mathbb R^2\to\mathbb R^2$ be the counterclockwise rotation of angle $\theta$, given explicitly by $T(x,t)=(x\cos\theta-t\sin\theta,x\sin\theta+t\cos\theta)$. Letting $v=u\circ T^{-1}$, we get that $v$ is continuous on $\{(x,t)\in\mathbb R^2; \,\alpha x \leq t\}$, $v(x\cos\theta,x\sin\theta)=u_0(x)$ and $u_R^*(x)=\sup_{t>x\sin\theta}v(x\cos\theta,t)$ for any $x\in\mathbb R$. Since rotations preserve harmonicity, if $t>x_0\sin\theta$ and $r<(t-x_0\sin\theta)\cos\theta$ we have 
\begin{align}\label{Renan_eq_nice}
v(x_0\cos\theta,t)=\frac{1}{\pi r^2}\int_{B_r(x_0\cos\theta,t)} v(y,s)\,\d y\,\d s\le\frac{1}{\pi r^2}\int_{-r}^r2\sqrt{r^2-y^2}\,\,u_R^*\left(\frac{x_0\cos\theta+y}{\cos\theta}\right)\d y.
\end{align} 
Since we are assuming that $x_0 \in A$ and $u^*(x_0) = u_R^*(x_0) > u_0(x_0)$, by the continuity of $v$ there exists a $\delta = \delta(x_0)$ such that 
\begin{align*}
v(x_0\cos\theta,t)< u^*(x_0) - \tfrac12 (u^*(x_0) - u_0(x_0))
\end{align*}
for $x_0\sin\theta < t<x_0\sin\theta+\delta$. Hence the supremum in $u^*(x_0) = u_R^*(x_0)=\sup_{t>x_0\sin\theta}v(x_0\cos\theta,t)$ can be restricted to times $t \geq x_0\sin\theta+\delta$, and we can choose any $r< \delta \cos \theta$ in \eqref{Renan_eq_nice} to get
\begin{align*}
u^*(x_0)\le\frac{1}{\pi r^2}\int_{-r}^r2\sqrt{r^2-y^2}\,\,u^*\left(x_0+\frac{y}{\cos\theta}\right)\d y
\end{align*}
and this implies the existence of $\varepsilon<\frac{r}{\cos\theta}$ verifying \eqref{claim_1_Renan}.

\smallskip

\noindent {\it Step 2}. If $u^*$ were not subharmonic (i.e. convex in each connected component), we would be able to find an interval $[a,b]\subset A$ such that $u^*(a)+u^*(b)<2u^*(\frac{a+b}{2})$. Let $h(x)=\frac{x-a}{b-a}u^*(b)+\frac{b-x}{b-a}u^*(a)$. Then $u^*-h$  vanishes at the endpoints $a$ and $b$ but is positive at their arithmetic mean. Choose $x_0\in[a,b]$ as small as possible such that $(u^*-h)(x_0)=\sup_{x\in[a,b]}(u^*-h)(x)$. Then for all $\varepsilon$ sufficiently small, 
$$(u^*-h)(x_0+\varepsilon)+(u^*-h)(x_0-\varepsilon)<2(u^*-h)(x_0),$$ 
which contradicts \eqref{claim_1_Renan}. This completes the proof.
\end{proof}

\begin{lemma}[Reduction to the Lipschitz case - non-tangential version]
In order to prove parts {\rm (i)} and {\rm (iii)} of Theorem \ref{Thm5} it suffices to assume that the initial datum $u_0:\R \to \R^+$ is Lipschitz. 
\end{lemma}
\begin{proof}
It is the same as the proof of Lemma \ref{Red_Lip_case}, replacing identity \eqref{Red_pf_eq_sem_pro} with 
$$u_{\varepsilon}^*(x)=\sup_{\substack{t >0 \\ |y-x|\le \alpha t}}P(\cdot,t)*u_\varepsilon(y)=\sup_{\substack{t >0 \\ |y-x|\le \alpha t}} u(y,t+\varepsilon).$$ 
Note that $u_\varepsilon^*\to u^*$ pointwise as $\varepsilon\to0$.
\end{proof}

\subsection{Proof of Theorem \ref{Thm5}} Once we have established the lemmas of the previous subsection, together with Lemma \ref{Lem7_Renan}, the proof of Theorem \ref{Thm5} follows essentially as in the proof of Theorem \ref{Thm1}. We omit the details.

\subsection{A counterexample in higher dimensions} If $\alpha>0$ and $d >1$, the non-tangential maximal function \eqref{def_max_function_cone} in $\R^d$ is not necessarily subharmonic in the detachment set. We now present a counterexample.

\smallskip 

Recall the explicit form of the Poisson kernel $P(x,t)$ as defined in \eqref{Intro_Poisson}. Let $u_0: \R^d \to \R$ be given by
\begin{equation*}
u_0(x) = (1 + |x|^2)^{\frac{-d +1}{2}} = (d-1) \int_1^{\infty} \frac{s}{(s^2 + |x|^2)^{\frac{d+1}{2}}}\,\ds.
\end{equation*}
Writing $C_d = \Gamma \left(\frac{d+1}{2}\right)\pi^{-(d+1)/2}$ we get
\begin{align*}
u(x,t) & = \int_{\R^d}  P(x-y,t)\,u_0(y)\,\dy\\
& = \frac{(d-1)}{C_d} \int_{\R^d} \int_{1}^{\infty} P(x-y,t)\,P(y,s)\,\ds\,\dy\\
& =  \frac{(d-1)}{C_d}  \int_{1}^{\infty} \int_{\R^d} P(x-y,t)\,P(y,s)\,\dy\,\ds\\
& = \frac{(d-1)}{C_d}  \int_{1}^{\infty} P(x, t+s)\,\ds\\
& = \big( (t+1)^2 + |x|^2\big)^{\frac{-d +1}{2}}.
\end{align*}
This is a translation of the fundamental solution of Laplace's equation on $\R^{d+1}$. A direct computation yields
\begin{align*}
u^*(x) = 
\left\{
\begin{array}{ll}
u_0(x) \ & \ {\rm if} \ |x| \leq \tfrac{1}{\alpha};\\
\left( \frac{(\alpha + |x|)^2}{\alpha^2 +1}\right)^{\frac{-d +1}{2}}& \ {\rm if} \ |x| > \tfrac{1}{\alpha}.
\end{array}
\right.
\end{align*}
From this we obtain 
\begin{equation*}
-\Delta u^*(x) = (d-1)\,\frac{(\alpha^2 +1)^{\frac{d-1}{2}}}{(\alpha + |x|)^{d+1}}  \left( \frac{\alpha}{|x|}(d-1) - 1\right)
\end{equation*}
for $|x| > \tfrac{1}{\alpha}$. This is strictly positive (hence $u^*$ is superharmonic) for $\tfrac{1}{\alpha} < |x| < (d-1) \alpha$ (assuming that this interval is nonempty, i.e. that $(d-1) \alpha^2 >1$).

\section*{Acknowledgements}
E.C. acknowledges support from CNPq-Brazil grants $305612/2014-0$ and $477218/2013-0$, and FAPERJ grant $E-26/103.010/2012$.

\end{document}